\documentclass[article,ij4uq]{ij4uq_adj}              %  final version
\usepackage[hang]{footmisc}
\usepackage{graphicx}
\usepackage[]{color}
\usepackage[]{tikz} %\usetikzlibrary{matrix, chains, arrows}https://www.overleaf.com/project/5e875960091fe800018d8a72
\usepackage{verbatim}
\usepackage{amsmath}
\usepackage[margin=3cm]{geometry}
\usepackage{natbib}
\usepackage{rotating}
\usepackage{bbold}
\usepackage{placeins}
\usepackage{multirow}
\usepackage{booktabs}
\usepackage{colortbl}
\usepackage{url}
\usepackage[official]{eurosym}
\usepackage{amsfonts}

\usepackage{arial}

\newtheorem{example}[theorem]{Example}
\newtheorem{thm}{Theorem}[section]

\newtheorem{lem}[thm]{Lemma}

\theoremstyle{definition}
\newtheorem{defn}[thm]{Definition}
\theoremstyle{remark}

\numberwithin{equation}{section}
\newcommand{\bb}[1]{\boldsymbol{#1}}
\newcommand{\signedpermutation}{
\pgfmatharray{\sp}{0}\let\n\pgfmathresult
\pgfmathparse{360.0/\n}\let\segment\pgfmathresult
\pgfmathparse{\segment/2}\let\shift\pgfmathresult
\def\radius{1.5cm}
\def\labelrad{1.7cm}
\def\regionboundaryin{1.4cm}
\def\regionboundaryout{1.6cm}
\scalebox{1.3}{\begin{tikzpicture}
\foreach \x in {1,2,...,\n}
{
  \draw[thick] (360-\x*\segment+90:\regionboundaryin)
             --(360-\x*\segment+90:\regionboundaryout);
  \pgfmatharray{\sp}{\x}\let\tmp\pgfmathresult 
  \pgfmathparse{int(\tmp)}\let\tmp\pgfmathresult 
  \node at (360-\x*\segment+90+\shift:\labelrad) {\tmp};
  \pgfmathparse{360-(\x-1)*\segment+90}\let\alpha\pgfmathresult;
  \pgfmathparse{360-(\x-1)*\segment+90-\segment}\let\beta\pgfmathresult;
  \pgfmathgreater{\tmp}{0}\let\decision\pgfmathresult
  \ifnum \decision=1
    \draw[color=black,very thick,>=latex,->] (\alpha:\radius) arc (\alpha:\beta:\radius);
  \else
    \draw[color=black,very thick,>=latex,->] (\beta:\radius) arc (\beta:\alpha:\radius);
  \fi
};
\end{tikzpicture}}}
\newcommand{\unsignedpermutation}{
\pgfmatharray{\sp}{0}\let\n\pgfmathresult
\pgfmathparse{360.0/\n}\let\segment\pgfmathresult
\pgfmathparse{\segment/2}\let\shift\pgfmathresult
\def\radius{1.5cm}
\def\labelrad{1.7cm}
\def\regionboundaryin{1.4cm}
\def\regionboundaryout{1.6cm}
\scalebox{1.3}{\begin{tikzpicture}
\foreach \x in {1,2,...,\n}
{
  \draw[thick] (360-\x*\segment+90:\regionboundaryin)
             --(360-\x*\segment+90:\regionboundaryout);
  \pgfmatharray{\sp}{\x}\let\tmp\pgfmathresult 
  \pgfmathparse{int(\tmp)}\let\tmp\pgfmathresult 
  \node at (360-\x*\segment+90+\shift:\labelrad) {\tmp};
  \pgfmathparse{360-(\x-1)*\segment+90}\let\alpha\pgfmathresult;
  \pgfmathparse{360-(\x-1)*\segment+90-\segment}\let\beta\pgfmathresult;
  \pgfmathgreater{\tmp}{0}\let\decision\pgfmathresult
  \ifnum \decision=1
    \draw[color=black,very thick] (\alpha:\radius) arc (\alpha:\beta:\radius);
  \else
    \draw[color=gray,very thick] (\beta:\radius) arc (\beta:\alpha:\radius);
  \fi
};
\end{tikzpicture}}}
\fancypagestyle{plain}{%
  \fancyhf{}
  \fancyfoot[R]{\small\bf\thepage }
  \fancyfoot[L]{\fottitle}
  }
\begin{document}
\volume{}
\title{Majorisation as a theory for uncertainty}
\titlehead{}
\authorhead{}
%For at least  authors with different addresses, use instead the following commands
\corrauthor[1]{Victoria Volodina}
\author[1]{Nikki Sonenberg}
\author[2]{Edward Wheatcroft}
\author[2]{Henry Wynn}

\corremail{v.volodina@turing.ac.uk}
\corraddress{The Alan Turing Institute, 96 Euston Road, London NW1 2DB}
\address[1]{The Alan Turing Institute, 96 Euston Road, London NW1 2DB}
\address[2]{London School of Economics and Political Science, Houghton Street, London, WC2A 2AE}
% End information for at least  authors with different addresses
% For authors with the same post address,
%\corrauthor{First A. Author}
%\corremail{f.author@affiliation.com}
%\author{Second B. Author, Jr.}
%\address{Department of Chemistry and Courant, Institute of Mathematical Sciences, New York, NY 10012, USA}
% End commands for all authors with the same address

\setlength{\parindent}{0pt}

\dataO{mm/dd/yyyy}
%\dataO{}
\dataF{mm/dd/yyyy}
%\dataF{}

\abstract{
 
Majorisation, also called rearrangement inequalities, yields a type of stochastic ordering in which two or more distributions can be compared. In this paper we argue that majorisation is a good candidate as a theory for uncertainty.  
We present operations that can be applied to study uncertainty in a range of settings and demonstrate our approach to assessing uncertainty with examples from well known distributions and from applications of climate projections and energy systems.

 %Majorisation, also called rearrangement inequalities, provides a representation of the peakedness of probability distributions, which is related to the idea of uncertainty. This method also implies ordering on probabilities and distribution functions. We demonstrate that majorisation approach is dimension free by obtaining univariate decreasing rearrangements. Thus, we can consider the ordering of two (or more) distributions with different support. We also present operations, including inverse mixing and maximise/minimise to combine and analyse uncertainties associated with different distribution functions. We illustrate the performance of our approaches to empirical examples with applications to scenario analysis and simulations.
}


\maketitle
%\tableofcontents

\section{Introduction}\label{sec:intro}
Rearrangements manipulate the shape of a geometric object while preserving its size \cite{Burchard2009}. Majorisation arises from rearrangements and provides an order on probability vectors from which various inequalities follow, as first established by Hardy, Littlewood and Polya \cite{Hardy1988} which, in turn, led to key work by Marshall, Olkin and Arnold \cite{Marshall1979}. 
Applications of majorisation have appeared in diverse fields including economics \cite{Arnold2018, Lorenz1905}, chemistry \cite{Klein1997}, statistics \cite{Degroot1988,Giovagnoli1987,Pukelsheim1987}, and more recently quantum information \cite{Partovi2011}. 

The concept of majorisation yields a \emph{partial} ordering, not a \emph{total} ordering, that is there are vectors for which neither vector majorises the other, so they are not comparable. In contrast, consider the well known measure of uncertainty, Shannon entropy, that corresponds to the resources required to send information that will eliminate the uncertainty \cite{cover1999}. 
Such entropic measures of uncertainty impose a total ordering.
Majorisation expresses a form of uncertainty, as the word ``more" in the statement ``more uncertain'' can be interpreted as a statement of relative order through the majorisation partial order.
Further, this notion of uncertainty does not have the extra requirement that it relies on a measure of information \cite{jacobs2014},  and does not make any assumptions about its functional form \cite{friedland2013}.

When vectors cannot be compared with respect to the majorisation ordering,  questions of relative uncertainty are unanswerable: one would have to specify in what way one was more certain or uncertain, at which point one could select a specific order-preserving comparison function.  
The partial order is weaker in the mathematical sense but, if events {\em can} be compared, the comparison is stronger as it can be made for fewer pairs of events.  
This is not a shortcoming of majorisation, but rather a consequence of its rigorous approach to ordering uncertainty~\cite{Partovi2011}. 
 
The challenge of defining the meaning of uncertainty 
has long been recognised: in 1914 Bertrand Russell wrote ``These varying degrees of certainty attaching to different data may be regarded as themselves forming part of our data: they, along with the other data, lie within the vague, complex, inexact body of knowledge which it is the business of the philosopher to analyse'' \cite{russell1914our}.

We present two properties that we believe make majorisation a good candidate for a theory of uncertainty: (i) being dimension-free and (ii) being geometry-free. These properties create the possibility of comparing multivariate distributions with different support and different numbers of dimensions. 

Majorisation is, in a well-defined sense, \emph{dimension-free}. In this paper, we show how this approach enables us to create, for a multivariate distribution, a univariate decreasing rearrangement (DR) by considering a decreasing threshold and ``squashing'' all of the multivariate mass for which the density is above the threshold to a univariate mass adjacent to the origin. 
  
The \emph{geometry-free} property follows because majorisation is independent of the support of the distribution. This distinguishes the approach from metric based measures such as variance and various types of multidimensional generalised variances \cite{pronzato2018simplicial}.
Metric-based dispersion orderings are well known and discussed in, for example, \cite{bickel1979descriptive}, with multivariate versions in  \cite{belzunce2008multivariate,giovagnoli1995multivariate}. 
 
Our contribution is to introduce a set of operations that can be applied to study uncertainty in a range of settings. These operations include how to project from many dimensions into one, how to combine two probability distributions, and mixing uncertainties (with different weights). We believe that the form of uncertainty captured by majorisation is close in spirit to entropy, but that it is less restrictive. We illustrate the introduced operations with examples and demonstrate that entropic measures can be used together with majorisation.
 
This paper is organised as follows, in the remainder of this section we introduce the concept of majorisation for discrete probabilities and in Section \ref{sec:cont_major} we present results for the continuous case. In Section \ref{sec:multivariate} we present the key concepts for multivariate distributions. In Section \ref{sec:operations}, we collect together operations for the study of uncertainty and in Section \ref{sec:algebra} we define a lattice and an algebra for uncertainty. We present empirical applications in Section \ref{sec:empirical} and concluding remarks are given in Section~\ref{sec:conclusion}.

\subsection{Discrete majorisation and related work}\label{sec:discrete}
\begin{defn} \cite{Marshall1979} Consider two discrete distributions with $n$-vectors of
probabilities $p_1=(p^{(1)}_1,\ldots, p_n^{(1)})$  and $p_2=(p^{(2)}_1,\ldots, p_n^{(2)})$, where
$\sum_{i=1}^n p_i^{(1)}=\sum_{i=1}^n p_i^{(2)}=1$.
Placing the probabilities in decreasing (i.e., nonincreasing) order:
\begin{align}
\tilde{p}^{(1)}_1 \geq \ldots \geq \tilde{p}_n^{(1)}\quad \text{ and } \quad \tilde{p}^{(2)}_1 \geq \ldots \geq \tilde{p}_n^{(2)},
\end{align}
it is then said that $p_2$ majorises $p_1$, written $p_1 \preceq p_2$, if % when, for all $n$,
\begin{align}
\sum_{i=1}^k \tilde{p}_i^{(1)} \leq \sum_{i=1}^k \tilde{p}_i^{(2)}, \quad k=1, \dots, n-1.
\end{align}
\end{defn} 

This means that the largest element of $p_2$ is greater than the largest element of $p_1$ and the largest two elements of $p_2$  are greater than the largest two elements of $p_1$, etc. This implies that the distribution of $p_1$ is more \textit{disordered} or spread out than $p_2$, which, in turn, means that the distribution of $p_2$ has less uncertainty than that of $p_1$. For example, for any $p$, $(1/n, \dots, 1/n)\preceq p\preceq (1, 0, \dots, 0)$.% \cite{Canosa2006}.

Marshall \emph{et al.} \cite{Marshall1979} provided several equivalent conditions to $p_1\preceq p_2$. We present three (A1-A3) of the best known in detail below.\\
(A1) There is a doubly stochastic $n\times n$ matrix $P$, such that
\begin{align}\label{equiv:doubly}
p_1 = P p_2.
\end{align}
The intuition of this result is that a probability vector which is a mixture of the permutations of another is more disordered. The relationship between a stochastic matrix $P$ and the stochastic transformation function in the refinement concept was presented by DeGroot \cite{Degroot1988}. \\
(A2)  Schur \cite{Schur1923}  demonstrated that, if (A1) holds for some stochastic matrix $P$, then for all continuous convex functions $s( \cdot )$ and for all $n$,
\begin{align}\label{condition3}
\sum_{i=1}^n s(\tilde{p}_i^{(1)}) \leq  \sum_{i=1}^n s(\tilde{p}_i^{(2)}).
\end{align}

The sums in Equation (\ref{condition3}) are special cases of the more general Schur-convex functions on probability vectors.  In particular, information measures such as Shannon information, for which $s(y)=y\log(y)$, and Tsallis information, for which $s(y)=\frac{y}{\gamma}(y^{\gamma}-1), \gamma>0$, where, in the limit as $\gamma\rightarrow 0$, Shannon is obtained. 

For any function $h(y)$ of Shannon or Tsallis  entropies, where $h(y) = -s(y)$, $p_1  \preceq p_2$ implies $\sum_{i=1}^n h(\tilde{p}_i^{(1)}) \ge \sum_{i=1}^n h(\tilde{p}_i^{(2)})$, but not conversely. Further, if for every such function $h(y)$ this relationship holds, then  $p_1  \preceq p_2$.  (A2) indicates that the ordering imposed by majorisation is stronger than the ordering by any single entropic measure and, in a sense, is equivalent to all such entropic measures taken collectively \cite{Partovi2011}.

\noindent (A3) Let $\pi(p) = (p_{\pi(1)}, \ldots, p_{\pi(n)})$ be the vector whose entries are a permutation $\pi$ of the entries of a probability vector $p$, with symmetric group $S$, then 
\begin{align}
p_1 \in \mbox{conv}_{\pi \in S} \{\pi(p_2)\}.
\end{align}
That is to say, $p_1$ is in the convex hull of all permutations of entries of $p_2$. Majorisation is a special case of group-majorisation (G-majorisation) \cite{Eaton1977} for the symmetric (permutation) group \cite{Giovagnoli1985}. 
 
\section{Continuous majorisation}\label{sec:cont_major}
In this section, we define continuous majorisation, the analog of the partial sums definition in Section \ref{sec:discrete}, following Hardy \emph{et al.} \cite{Hardy1988}.

\begin{defn}
\label{drdefn}
Let $f(x)$ be a (univariate) pdf and define $m(y)=\mu \{z: f(z) \geq y\}$, where $\mu$ is Lebesgue measure. The decreasing rearrangement of $f(x)$ is
\begin{align}
\tilde{f}(z)=\mbox{sup}\{t: m(t) >z\},\; z >0.
\end{align}
\end{defn}

\begin{defn}
Let $\tilde{f}_1(z)$ and $\tilde{f}_2(z)$ be the DR of two pdfs $f_1(x)$ and $f_2(x)$, respectively and $\tilde{F}_1(z)$ and $\tilde{F}_2(z)$ their corresponding cdfs. We say that $f_2(x)$ majorises $f_1(x)$, written
$f_1 \preceq f_2$, if and only if
\begin{align} 
\tilde{F}_1(z) \leq  \tilde{F}_2(z),\;  z > 0.
\end{align}
\end{defn}

Similarly to the discrete case, we give three equivalent conditions for continuous majorisation.\\
\noindent (B1) For some non-negative doubly stochastic kernel $P(x,t)$,
\begin{align}
f_1(x) = \int P(x,t) f_2(t) dt.
\end{align}
\noindent (B2)  For all continuous convex functions $s(\cdot)$,
\begin{align} \label{equiv_convex}
\int s(f_1(z)) dz \leq \int s(f_2(z))dz.
\end{align}
\noindent (B3) Slice condition:  
\begin{align}
\int(f_1(x)-c)_+ dx \leq \int(f_2(x)-c)_+dx, \quad c>0. \label{eq:slice}
\end{align}

\begin{example}\label{example|_beta}
Consider the Beta$(3,2)$ distribution with pdf $p(z)=12(1-z)z^2$. We look for $z_1$ and $z_2$, where $z_1<z_2$, such that $p(z_1)=p(z_2)=c$, illustrated in
Figure~\ref{fig:newbetaplots} (left panel). Setting $z=z_2-z_1$, we have
\begin{align}\label{example_beta_system}
\begin{cases}
p(z_1)=12(1-z_1)z_1^2=y ,\\
p(z_2)=12(1-z_2)z_2^2=y,\\
z_2-z_1=z,\\
0\leq z\leq 1,
\end{cases}
\end{align}
from which the DR can be obtained by eliminating $z_1$ and $z_2$ and setting
$\tilde{f} = y$. With the elimination variety, a set of points (solutions) satisfying a system of polynomial equations being equal to zero,  
$ 48z^6 - 96z^4 + 9y^2 + 48z^2 - 16y =0$, we obtain
\begin{align}
\tilde{f}(z) = \left\{ 
 \begin{array}{l}  
\frac{8}{9} 
+ \frac{4}{9} (-27z^6 + 54z^4 - 27z^2 + 4)^{\frac{1}{2}}, \quad 0 \leq z \leq \frac{1}{\sqrt{3}},  \\                                              \frac{8}{9} -  \frac{4}{9} (-27z^6 + 54z^4 - 27z^2 + 4)^{\frac{1}{2}}, \quad \frac{1}{\sqrt{3}} \leq z \leq 1.
\end{array}
\right.
\end{align}

The DR cdf is obtained by adjoining the equations $Y = F(z_2)-F(z_1)$ to get the second variety $3z^8-12z^6+16z^4+9Y^2-16Yz = 0$, then
\begin{equation}
\tilde{F}(z) = \frac{z}{9} \big(\sqrt{-(3z^2-4)^3} +8 \big),
\end{equation}
and is illustrated in Figure \ref{fig:newbetaplots} (right panel) alongside the pdf (central panel).
 
 \begin{figure}[h!]
\begin{center}
\includegraphics[height=4.5cm]{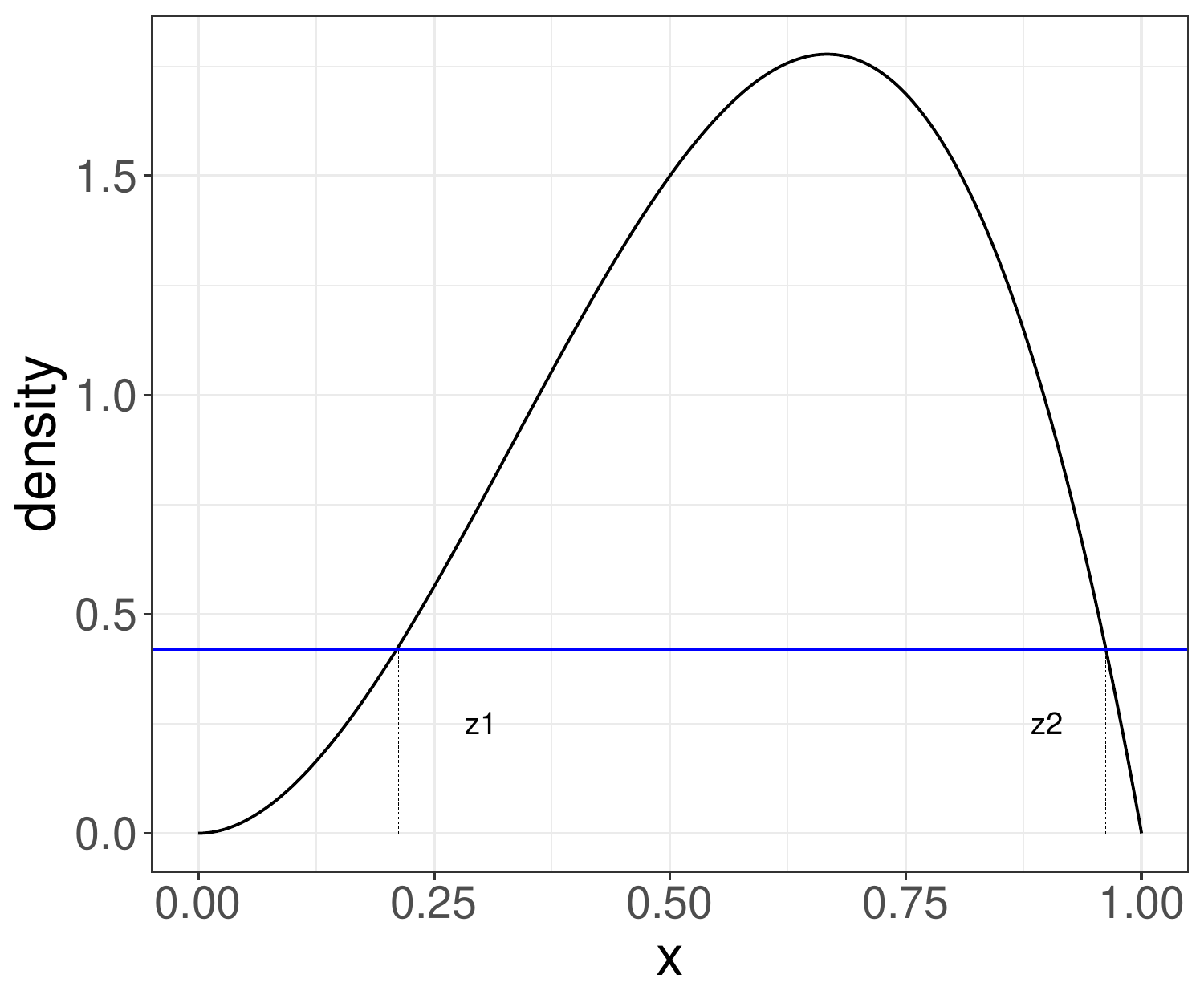}
\includegraphics[height=4.5cm]{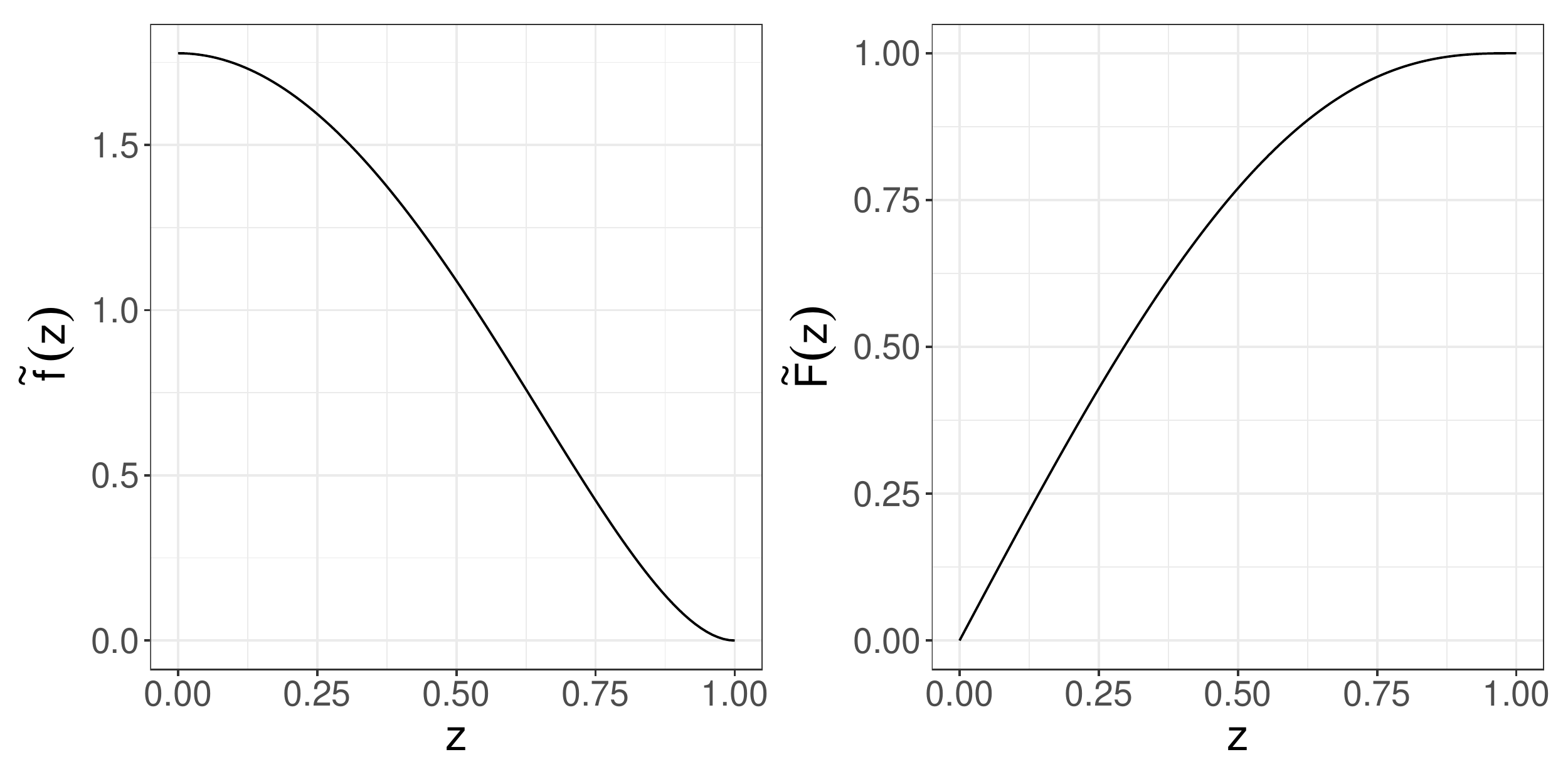}
\caption{Example \ref{example|_beta}. \emph{Left panel:} Identification of $z_{1}$ and $z_{2}$,  \emph{central panel}: DR pdf $\tilde{f}(z)$,  \emph{right panel}: DR cdf $\tilde{F}(z)$.}
\label{fig:newbetaplots}
\end{center}
\end{figure}
\end{example}

It is hard to derive the DR in the general case when $f_1\preceq f_2$ in which $f_i(x) \sim \mbox{Beta}(a_i,b_i),\; i=1,2$. However, we can prove the following.

\begin{lem}
Assume $a_1 , b_1 , a_2 , b_2 > 1$. If pdfs $f_1 (x)\sim \text{Beta}(a_1 , b_1 )$ and $f_2 (x)\sim\text{Beta}(a_2 , b_2 )$, have the same mode, then  $\max_x  f_1(x) \leq  \max_x f_2(x)$ if and only if $X_1 \preceq X_2.$
\end{lem}

\begin{proof} We first prove that, under the same mode condition, $f_1(x)$ and $f_2(x)$ intersect at two distinct $x$-values at which the values of $f_1(x)$ and $f_2(x)$ are the same. Setting the modes equal, 
$$ \frac{a_1-1}{a_1+b_1-2} = \frac{a_2-1}{a_2+b_2-2},$$
without loss of generality, we have that $a_2  > a_1$ and find
$$\frac{f_1(x)}{f_2(x)} = \left\{x (1-x)^u \right\}^v C, $$
where $u= a_2-a_1,v = \frac{b_1-1}{b_2-1}$ and $C$ is a constant. Setting this equal to 1, we have two solutions given by
$ x (1-x)^u = C^{-\frac{1}{v}}$. It is then straightforward to verify that the common value of $f_1(x)$ and $f_2(x)$ is the same at the two solutions. 
The proof is completed by using the slice condition in Equation \eqref{eq:slice}.
\end{proof}

Comparison of DR cdfs may use algebraic or numerical techniques. If the DR cdfs are polynomial, then comparison involves testing whether two increasing polynomials cross or one dominates the other over the union of the support of the distribution. Whether closed form characterisations of the DR are available for non-polynomial cdfs is outside the scope of this paper.

\section{Multivariate case: matching of uncertainty}\label{sec:multivariate}
 \begin{defn}\label{multitouni} 
A univariate decreasing rearrangement $\tilde{f}(z)$, compatible with $f(x)$, is, for all constants $c\geq 0$,
\begin{align}
\label{decreaseRearrangemult}
\int_{\{x:f(x)\geq c \}}f(x)dx=\int_{\{z:\tilde{f}(z)\geq c\}}\tilde{f}(z)dz.
\end{align}
\end{defn}
\begin{proof}
As
 \begin{align}
     {\{x:f(x)\geq c \}} = {\{z:\tilde{f}(z)\geq c\}},
\end{align}
then the volume of these sets are consistent \cite{Burchard2009}.
\end{proof}
This result induces a one dimensional DR from a multidimensional distribution. The following lemma is a key result and shows that the information/entropy for $X \sim f(x)$ and $Z \sim \tilde{f}(z)$ are the same. 

\begin{lem}
Let $f(x)$ be a multidimensional pdf and $\tilde{f}(z)$ on $[0, \infty]$ its decreasing rearrangement. Then, given a convex function $\varphi(x)$, we have
\begin{align}
\int_S \varphi(f(x))  dx = \int_0^{\infty} \varphi(\tilde{f}(z)) dz,
\end{align}
where $S$ is the support of $f(x)$.
\begin{proof} 
Matching volume to length elements in $S$ and $[0,1)$, for  $c>0$ and small $\delta c  > 0$ we have
$$ \int_{x: f(x) \geq c, x \in S} f(x)dx -  \int_{x: f(x) \geq c + \delta c, x \in S} f(x) dx =  \int_{z: \tilde{f}(z) \geq c, z \in [0, \infty)} \tilde{f}(z) dz  -\int_{z: \tilde{f}(z) \geq c +\delta c, z \in [0, \infty)} \tilde{f}(z) dz.$$
We can then write, approximately,
$$u(c)A(c, \delta c)=u(c)L(c, \delta c),$$
where $A(c, \delta c)$ and  $L(c, \delta c)$ are the corresponding increments in volume and length, respectively, as corresponding to the interval $[c,c+\delta c)$, that is
$ f^{(-1)}([c,c+ \delta))$ and $ \tilde{f}^{(-1)}([c,c+ \delta))$, respectively. Cancelling $c$, we can equate  $A(c, \delta c)$ and  $L(c, \delta c)$, and this allows us to recapture and equate the integrals of any measurable function $u(\cdot)$. In particular, we can write $u(c) = \varphi(f(c)).$
\end{proof}
\end{lem}

In Examples \ref{multi_norm_ex} and \ref{indep_exp_ex}, we demonstrate how to obtain a DR from a multivariate  distribution. The following idea is used to carry out computations: there may be cases in which, for a given $c$, the inverse set $f^{(-1)}(c)$ is described by some useful quantity $\delta$. Moreover $\delta$, expressed as a function of $x$, then becomes a random variable with a known (univariate) distribution. Since $\tilde{F}(\tilde{f}^{-1}(c) ) =F_{\delta}(f_{{X}}^{-1}(c) )$, Definition \ref{multitouni} can be expressed as
\begin{align} \label{eq:valid_DR}
\tilde{f}(r) & =  f_{\delta} \left( f_X^{(-1)}(\tilde{f}(r)) \right) \frac{\partial}{\partial r}\left(f_X^{(-1)}(\tilde{f}(r))\right).
\end{align}
 
\begin{example}\label{multi_norm_ex}
Let the random vector ${X}=(X_1, \dots, X_n)^T$ be an $n$-variate standard normal distribution with pdf
\begin{align}
f_{{X}}(x_1, \dots, x_n)=\frac{1}{(2\pi)^{\frac{n}{2}}}\exp\Big\{-\frac{1}{2}\sum_{i=1}^n x^2_i \Big\}.
\end{align}
We refer to ${X}$ as a spherical Gaussian random vector with ${X}\sim\text{N}_n({0}, I_n)$, where ${0}$ is an $n$-vector of zeros and $I_n$ is the $n\times n$ identity matrix. To construct the DR, we slice first the pdf at $f_{{X}}(x_1, \dots, x_n)=c$. We have that the square of the radius of a spherical Gaussian random vector is $R^2 = \sum_{i=1}^nX_i^2$, defining $r^2 =\sum_{i=1}^n x_i^2$, then
    \begin{align}
    r=\Big(-2\log\big((2\pi)^{n/2} c\big) \Big)^{1/2},
\end{align}
where the volume of the $n$-dimensional Euclidean ball of radius $r$ is
\begin{equation}
    \label{eq:ball_volume}
    V_n(r)=\frac{\pi^{n/2}}{\Gamma\Big(\frac{n}{2}+1 \Big)}r^n,
\end{equation}
from which we obtain
\begin{align}
c=\frac{1}{(2\pi)^{n/2}}\exp\bigg\{-\frac{1}{2}\bigg(\frac{V_n(r)\Gamma(n/2+1)}{\pi^{n/2}} \bigg)^{2/n} \bigg\},
\end{align}
noting the values of $c$ and $V_n(r)$ are dependent on each other. To generalise this expression, we replace $c$ and $V_n(r)$ with $\tilde{f}(z)$ and $z$, respectively. The resulting form of the DR is
\begin{equation}
\label{eq:DRM_mult_normal}
\tilde{f}(z) = \frac{1}{(2\pi)^{n/2}}\exp\bigg\{-\frac{1}{2}\bigg(\frac{z}{V_n} \bigg)^{2/n} \bigg\},
\end{equation}
where $V_n$ is the volume of the unit sphere in $\mathbb{R}^n$. For the two-dimensional multivariate normal, we illustrate the construction of the univariate DR in Figure \ref{fig:DRM2}.

We can validate the form of the DR in Equation (\ref{eq:DRM_mult_normal}) by the construction in Equation (\ref{eq:valid_DR}), where $R^2=\sum_{i=1}^n X_i^2$ follows a Chi-squared distribution with $n$ degrees of freedom.

\begin{figure}[h!]
\begin{center}
\includegraphics[height=0.2\textheight]{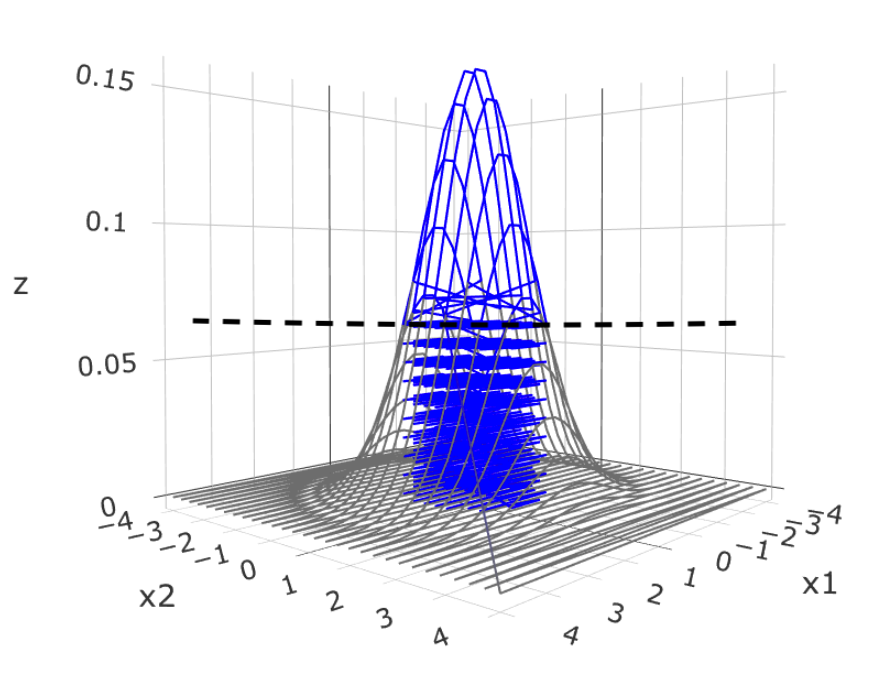}\includegraphics[height=0.2\textheight]{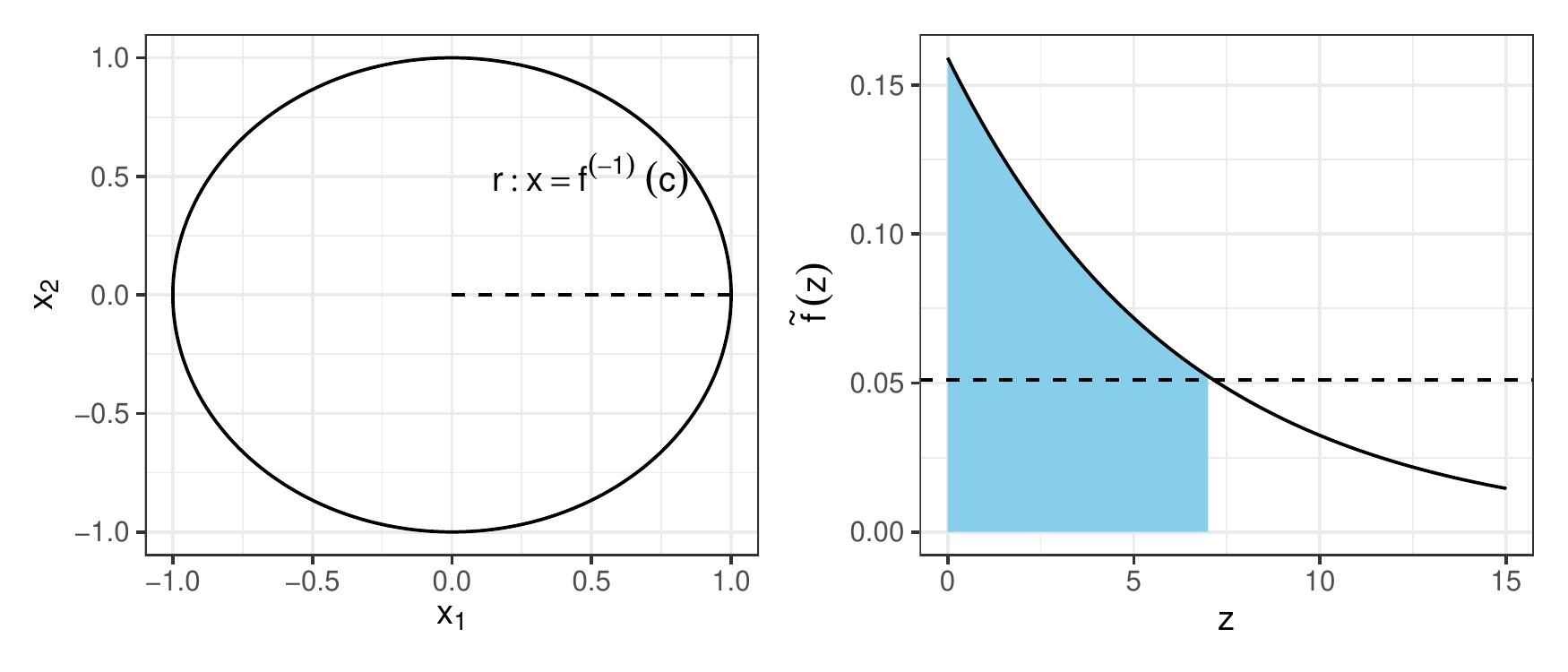}
 \end{center}
\caption{Example \ref{multi_norm_ex}. \textit{Left panel:} Density plot of a two-dimensional standard multivariate normal. The dashed line and blue shaded region correspond to $f(x)=c$ and $\int_{\{x:f(x)\geq c \}}f(x)dx$, respectively. 
\textit{Central panel:} A plot to demonstrate the connection between ${X}=(X_1, X_2)^T$ and $R^2$. The radius $r$ of a blue circle corresponds to $x=f^{-1}(c)$.
\textit{Right panel:} The DR $\tilde{f}(z)$ obtained for the multivariate normal. The blue shaded region corresponds to $\int_{\{z: \tilde{f}(z)\geq c\}}\tilde{f}(z) dz$.}
\label{fig:DRM2}
\end{figure}
 
 \end{example}

\begin{example}\label{indep_exp_ex}
Consider the $n$-fold independent standard exponential distribution with pdf
 \begin{align}
f_X(x_1, \dots, x_n)=\exp\left\{-\sum_{i=1}^n x_i \right\}.
\end{align}
As $f_X(x_1, \dots, x_n) = f_1(x_1) f_2(x_2) \cdots f_n(x_n)$, slicing the pdf at $c=f_{{X}}(x_1, \dots, x_n)$ yields  $-\log(c)=\sum_{i=1}^n x_i$.
The volume of an $n$-dimensional simplex  in which all $n$ variables are greater than $0$ but with sum less than $R$ is $V_n =  R^n/n!$, then $ c=\exp\left\{-(n!V_n)^{1/n} \right\}$. Replacing $c$ and $V_n$ with $\tilde{f}(z)$ and $z$, respectively, the DR can be written as
\begin{equation}
\label{eq:DRM_mult_exp}
\tilde{f}(z)=\exp\left\{-(n!z)^{1/n} \right\}.
\end{equation}
To verify the form of the DR in Equation (\ref{eq:DRM_mult_exp}), we can use Equation (\ref{eq:valid_DR}) with the relationship $\delta$ given by $R=\sum_{i=1}^n X_i$, such that $R\sim\text{Gamma}(n, 1)$.
\end{example}

\begin{example}\label{example_multi}
Figure \ref{fig:Plot3} shows the DR cdfs for standard normal and exponential distributions with $n=1, 2, 3, 4$. We conclude that for  $X\in\mathbb{R}^n$ and $Y\in\mathbb{R}^m$, where $X\sim f_X$ and $Y\sim f_Y$, and $f_X$ and $f_Y$  come from the same family of multivariate densities, standard normal or standard exponential, when $n>m$, $f_X\preceq f_Y$. We observe that univariate distributions majorise the remaining cdfs, which implies that adding more random variables drives the uncertainty up.  
\end{example}

\begin{figure}[h!]
\begin{center}
\includegraphics[width = 0.8\textwidth]{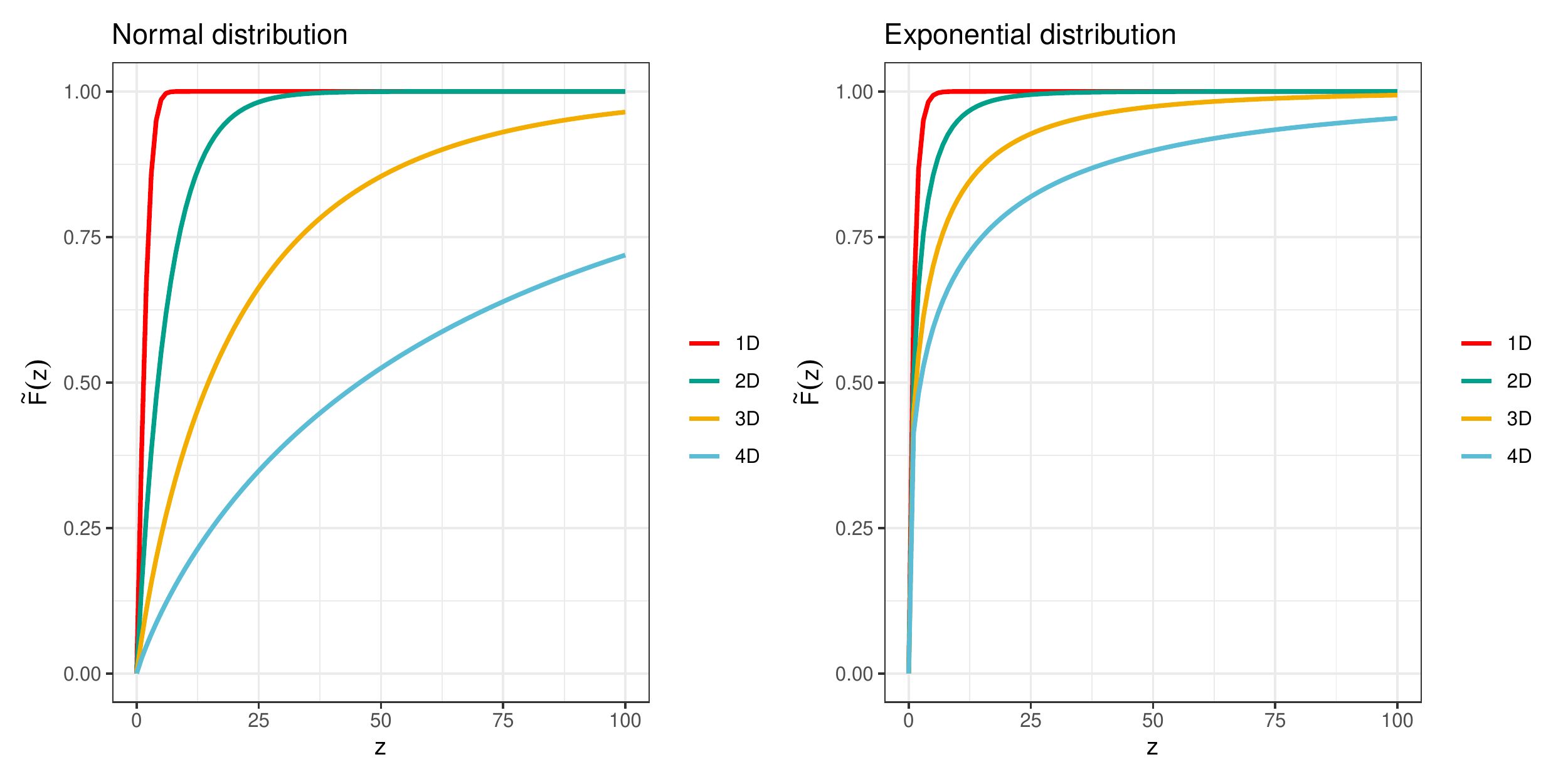}
\caption{Example \ref{example_multi}. \textit{Left panel:} $\tilde{F}(z)$ for the standard normal with $n=1,2,3,4$. \textit{Right panel:}  $\tilde{F}(z)$ for the independent standard exponential distribution with $n=1,2,3,4$.}\label{fig:Plot3}
\end{center}
\end{figure}

\section{Some operations with $\preceq$}\label{sec:operations}
We present operations for combining the uncertainty of two distributions and discuss the effect of dependence and the area of support of the distributions on the ordering. In Section \ref{subsec:InverseMix}, we introduce inverse mixing for discrete and continuous probability distributions. In Section \ref{subsec:dependece_order}, we discuss the use of majorisation ordering as an ordering of dependence for  distributions with the same marginal densities. In Section \ref{subsec:volume}, we present a volume-contractive mapping for a random variable which leads to a reduction in uncertainty.

\subsection{Inverse Mixing}
\label{subsec:InverseMix}
\begin{defn}\label{def:inversemixing}
Inverse mixing is defined by
\begin{align}
\tilde{f}_1\; [+] \;\tilde{f}_2=
\left(\tilde{f}_1^{(-1)}(z)+\tilde{f}_2^{(-1)}(z) \right)^{(-1)},
\end{align}
and $\alpha$-inverse mixing is defined by
\begin{align}
\tilde{f}_1\; [+]_{\alpha}\; \tilde{f}_2=\left(\tilde{f}_1^{(-1)}\left(\frac{z}{1-\alpha}\right) + \tilde{f}_2^{(-1)} \left(\frac{z}{\alpha}\right)\right)^{(-1)},
\end{align}
where  $0 < \alpha < 1$ is the mixing parameter.
\end{defn}

Inverse mixing is a method for combining uncertainty given two distributions over two different populations. We note that, when performing inverse mixing, the supports (atoms in the discrete case) can be different.
We demonstrate inverse mixing for the discrete and continuous distributions. For the continuous distribution, we consider cases in which the maximum values (modes) occur (i) at the same point, and (ii) at different points. 

\begin{example}\label{exampleworkplace}
Consider two distinct groups of people in a workplace. Define the probability of the $i^\text{th}$ member of group one obtaining a promotion by $p_i$ and, correspondingly, by $q_i$ for group two. Let  
$p = (0.577,\  0.192,\ 0.128, $ $ \ 0.064, \ 0.038)$ and $q = (0.730 ,\ 0.219, \ 0.036,\ 0.007,\ 0.007)$,  noting that 
$	p_1\geq p_2\geq\cdots\geq p_5$, $q_1\geq q_2\geq \cdots q_5$ and
$\sum_{i}p_i=1$, $\sum_{i}q_i=1$. 
To perform inverse mixing with $\alpha = \frac{1}{2}$, we take the inverse of each pmf, combine them and then sort them into ascending order (Figure \ref{fig:inversetilt}, left panel). The inverse is then taken to obtain a pmf (central panel). The inverse mixing procedure combines all of the probabilities scaled by a factor $\frac{1}{2}$, i.e. $\frac{1}{2} p_i \ \cup \ \frac{1}{2} q_i,$ and sorts them in decreasing order. The direct mixing procedure is obtained by summing the ordered probabilities of two populations and scaling them by a factor $\alpha$, i.e. $\frac{1}{2}(p_i+q_i)$ (right panel). Whilst both mixings provide information about the joint population, the inverse mixing also preserves information about the individual subpopulations.
\end{example}

\begin{figure}[h!]
\begin{center}
\includegraphics[width=.96\textwidth]{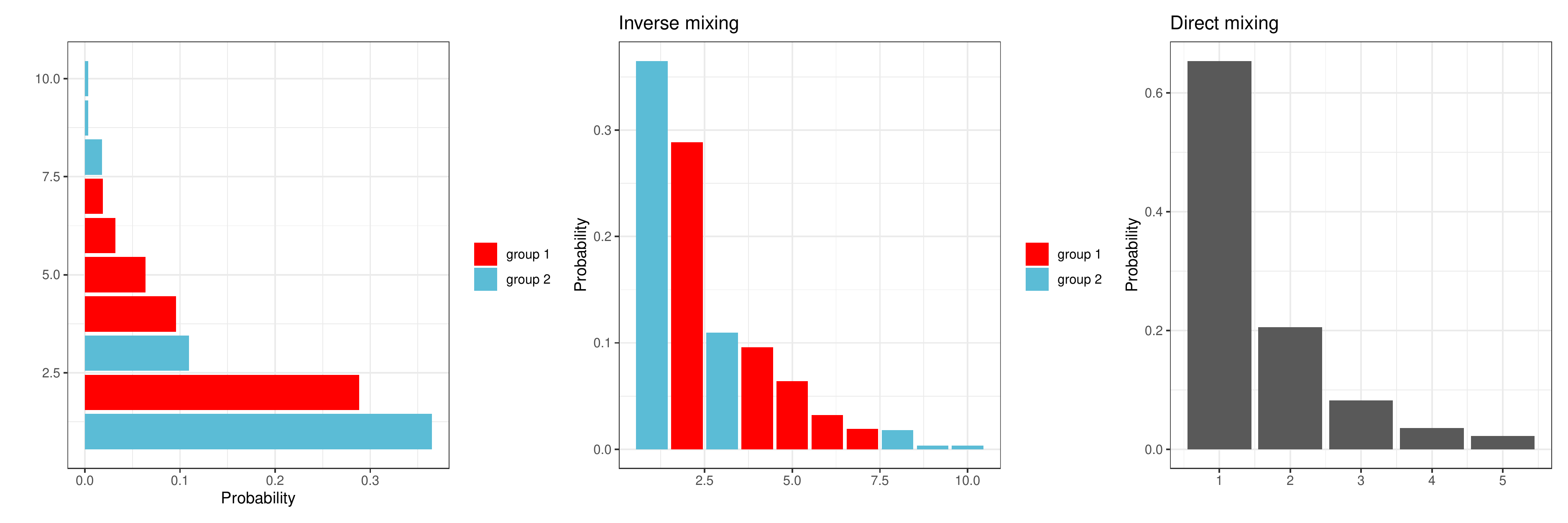}
\vspace{-2em}
\end{center}
\caption{Example \ref{exampleworkplace}. \textit{Left panel:} The addition of the two inverse pmfs, \textit{central panel:} inverse mixing distribution with $\alpha=\frac{1}{2}$, \textit{right panel:} direct mixing distribution with $\alpha=\frac{1}{2}$.
}
\label{fig:inversetilt}
\end{figure}

\begin{example}\label{example:unibi_exp}
For the univariate and bivariate exponential distributions, from Equation \eqref{eq:DRM_mult_exp} we have 
\begin{align}
\tilde{f}_1(z)=\exp\{-z\}, \quad \tilde{f}_2(z)=\exp\{-(2z)^{1/2}\}.
\end{align}
Observing that $0<\tilde{f}_1(z), \tilde{f}_2(z)\leq 1$, we have functional inverses, 
\begin{align}
\tilde{f}_1^{(-1)}(z) = -\log(z), \quad \tilde{f}_2^{(-1)}(z)=\frac{1}{2}(\log(z))^2, \quad z\in(0, 1],
\end{align}
and we illustrate these in Figure \ref{fig:inverse_mixing_con1} (left and central panels). The inverse mixing of the two distributions is 
\begin{align}
f^{(1)}(z)  = \left(-\log\Big(\frac{z}{1-\alpha}\Big) +\frac{1}{2} \Big(-\log\Big(\frac{z}{\alpha}\Big) \Big)^2 \right)^{(-1)} ,
\end{align}
for $0 \leq \alpha \leq 1$, and as the maximum values of the distribution functions occur at the $z=0$, the curve is smooth. The direct averaging of $f_1(x)$ and $f_2(x)$ is
\begin{align}
f^{(2)}(z) = \left((\alpha-1)\log(z)+\frac{\alpha}{2}(-\log(z))^2 \right)^{(-1)}.
\end{align}
For $\alpha=1/2$ we have
\begin{align}
f^{(1)}(z)=\Big(-\log(2z)+\frac{1}{2}\big(\log(2z)\big)^2 \Big)^{(-1)},
\end{align}
which is a quadratic in $\log(2z)$, so we obtain the two solutions $f^{(1)}(z)=\frac{1}{2}\exp\{1+\sqrt{1+2z}\}$ and $f^{(1)}(z)=\frac{1}{2}\exp\{1-\sqrt{1+2z}\}$. As the first solution does not integrate to one, we retain the second solution. The values of the mean, variance and Shannon entropy for $f^{(1)}(z)$ are 
$\left [\frac{7}{2}, \ \frac{99}{4}, \ \frac{3}{2}+\log(2)  \right]$, respectively. For the direct mixing, with  $\alpha=1/2$ and pdf $f^{(2)}(z)=\exp\left\{1-\sqrt{1+4z}\right\}$, 
the corresponding values are 
$\left [\frac{7}{4}, \ \frac{99}{16}, \ \frac{3}{2}\right]$. We observe that both the variance and the Shannon entropy are greater for $f^{(1)}(z)$ than for $f^{(2)}(z)$, which indicates that there is more uncertainty attributed to $f^{(1)}(z)$ than to $f^{(2)}(z)$. We confirm this finding with the comparison plot of inverse mixing and direct averaging in Figure \ref{fig:inverse_mixing_con1} (right panel). We have the relationship $f^{(2)}(z)=2f^{(1)}(2z)$, and can see $f^{(1)}(z)$ (red line) stretches the support of the distributions, and lowers the overall maximum, whereas $f^{(2)}(z)$ (blue line) preserves the maximum and shrinks the support.
\begin{figure}[h!]
\begin{center}
\includegraphics[width=1\textwidth]{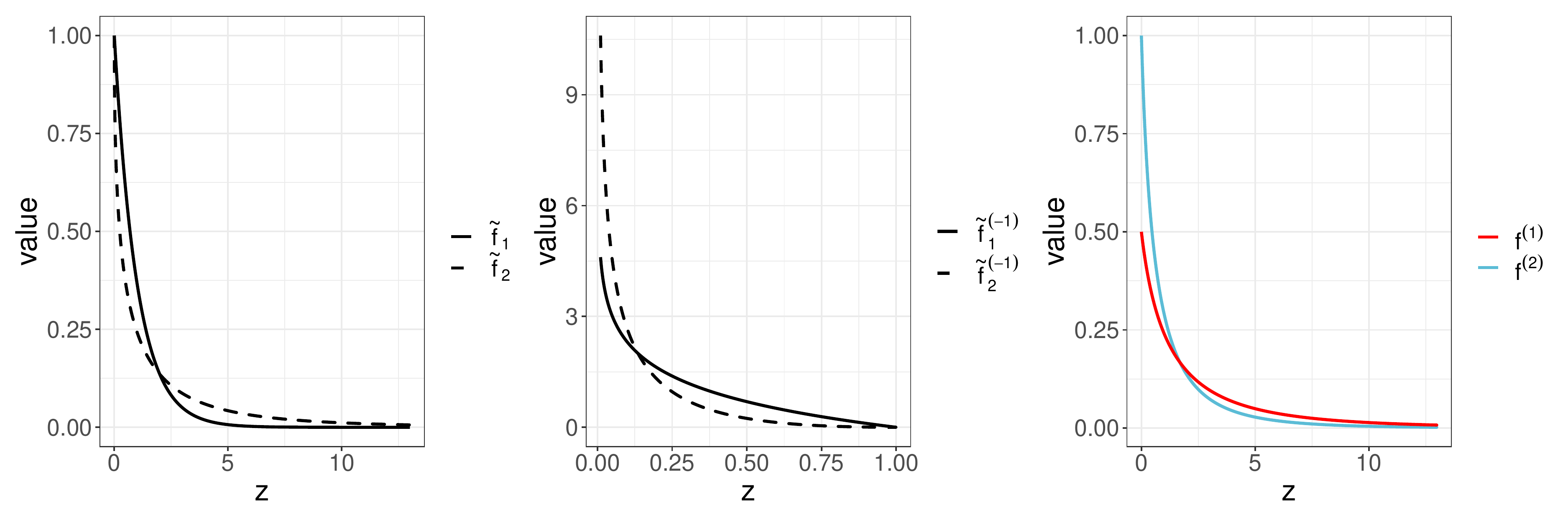}
\vspace{-2em}
\end{center}
\caption{\textit{Left panel:} DR functions $\tilde{f}_1(z)$ (solid line) and $\tilde{f}_2(z)$ (dashed line). \textit{Central panel:} functional inverses of the DR functions: $\tilde{f}_1^{(-1)}(z)$ (solid line) and $\tilde{f}_2^{(-1)}(z)$ (dashed line). \textit{Right panel:} inverse mixing and direct averaging: $f^{(1)}(z)$ (red line) and $f^{(2)}(z)$ (blue line).}
\label{fig:inverse_mixing_con1}
\end{figure}
\end{example}

\begin{example} \label{example:6}  
Consider exponential distributions with means 1 and 2, then the pdfs are already DRs,
\begin{align}
\tilde{f}_1(z)=\exp\{-z\}, \quad\tilde{f}_2(z)=\frac{1}{2}\exp\{-\frac{z}{2}\},
\end{align}
as illustrated in Figure \ref{fig:inverse_mixing_con2final} (left panel). Since
 $0<\tilde{f}_1(z)\leq 1$ and $0<\tilde{f}_2(z)\leq 1/2$, the functional inverses have different support:
\begin{align}
\tilde{f}_1^{(-1)}(z)=-\log(z), \quad z\in (0, 1], \quad \text{and} \quad  \tilde{f}_2^{(-1)}(z)=-2\log(2z), \quad z\in (0, 1/2].
\end{align}
For $\alpha=1/2$, the inverse mixing is
\begin{align}
f^{(1)}(z)=\tilde{f}_1\; [+]_{\frac{1}{2}}\; \tilde{f}_2=\left(-\log(2z)-2\log(4z) \right)^{(-1)}.
\end{align}
To avoid negative values of the expression inside the functional inverse, we propose the following modification:
\begin{align}
\tilde{f}_1^{(-1)}(2z)+\tilde{f}_2^{(-1)}(2z) =\max\{0, -\log(2z)\}+\max\{0, -2\log(4z)\},
\end{align}
illustrated in 
Figure \ref{fig:inverse_mixing_con2final} (central plot) which results in a kink at $z=0.25$. We take another functional inverse to obtain the inverse mixing, 
\begin{align}
f^{(1)}(z)&=\begin{cases}
\frac{1}{2}\exp\{-z\}, &\mbox{if } 0<z<\log(2) ,\\
\frac{1}{2}\exp\{ \frac{-2\log(2)-z}{3} \}, &\mbox{if } z \geq \log(2),
\end{cases}
\end{align}
and, as illustrated in Figure \ref{fig:inverse_mixing_con2final} (right panel), we observe a kink at $z=\log(2)$. 
For $\alpha=1/2$, the direct averaging of these distributions is
\begin{align}
f^{(2)}(z)=\left(-\frac{1}{2}\log(z)-\log(2z) \right)^{(-1)}.
\end{align}
To avoid negative values, we can write
\begin{align}
\frac{1}{2}\tilde{f}_1^{(-1)}(z)+\frac{1}{2}\tilde{f}_2^{(-1)}(z)=\max\left\{ 0,-\frac{1}{2}\log(z)\right\}+\max\left\{0, -\log(2z) \right\},
\end{align}
 illustrated by the solid line of Figure \ref{fig:inverse_mixing_con2final} (central plot), noting a kink at $z=\frac{1}{2}$. As a result, we obtain a kink in $f^{(2)}(z)$ at $z=-\frac{1}{2}\log{\frac{1}{2}}$ in Figure \ref{fig:inverse_mixing_con2final} (right panel). The direct averaging is
\begin{align}
f^{(2)}(z)
&=\begin{cases}
\exp\{-2z\}, &\mbox{if } 0<z<-\frac{1}{2}\log(\frac{1}{2}), \\
\exp\Big\{\frac{-2z-2\log(2)}{3}\Big\}, &\mbox{if } z\geq -\frac{1}{2} \log(\frac{1}{2}).
\end{cases}
\end{align}

The values of the mean, variance and Shannon entropy for $f^{(1)}(z)$ are 
$\left [ 2.85,  \ 8.91, \ 1 +  \frac{3}{2}\log(2)\right]$, respectively,
and corresponding values for $f^{(2)}(z)$ are 
$\left [1.42, \ 2.23, \ 1+\frac{1}{2}\log(2)\right]$. From the representation of inverse mixing and direct averaging in Figure \ref{fig:inverse_mixing_con2final}, we can see that $f^{(1)}(z)$ stretches the support of the distribution, whilst $f^{(2)}(z)$ shrinks it. The maximum (mode) from the direct averaging is double the maximum of the inverse mixing.

\begin{figure}[h!]
\begin{center}
\includegraphics[width=1\textwidth]{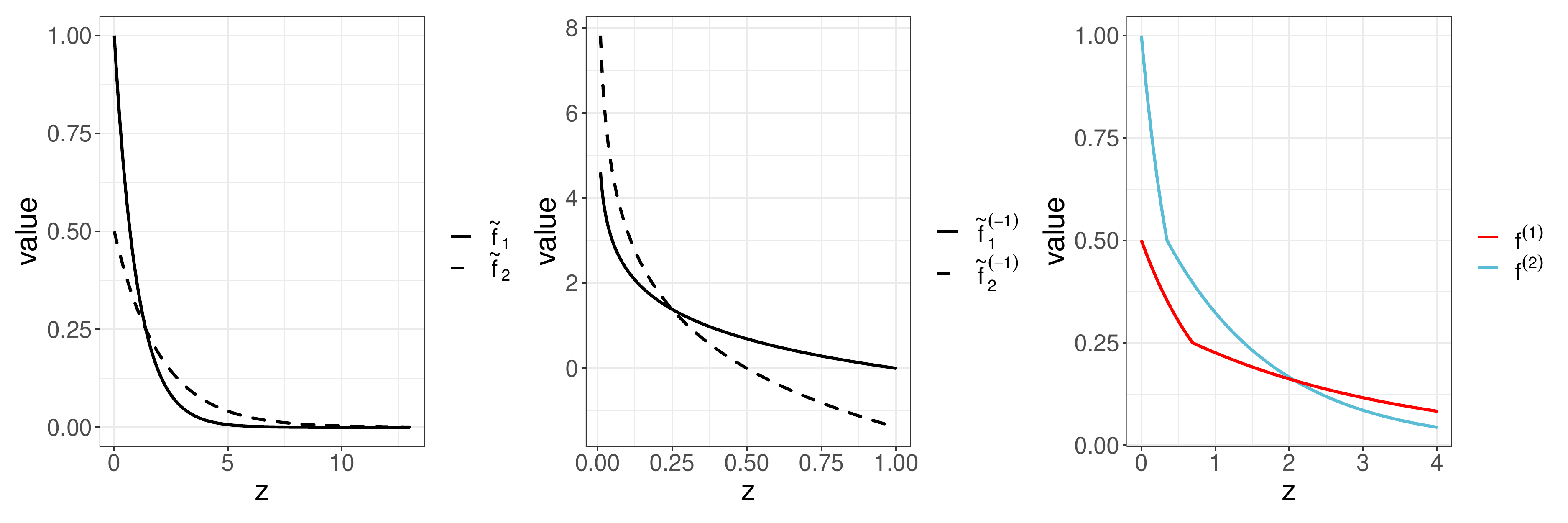}
\vspace{-2em}
\end{center}
\caption{\textit{Left panel:} DR functions $\tilde{f}_1(z)$ (solid line) and $\tilde{f}_2(z)$ (dashed line). \textit{Central panel:} functional inverses of the DR functions: $\tilde{f}_1^{(-1)}(z)$ (solid line) and $\tilde{f}_2^{(-1)}(z)$ (dashed line). \textit{Right panel:} inverse mixing and direct averaging: $f^{(1)}(z)$ (red line) and $f^{(2)}(z)$ (blue line).}
\label{fig:inverse_mixing_con2final}
\end{figure}
\end{example}

\subsection{Dependence ordering}
\label{subsec:dependece_order}
We briefly discuss how to obtain the ordering of dependence as well as measures of dependence using majorisation and entropy, on which Joe \cite{Joe1985,Joe1987} has written extensively.

\begin{defn}  Let $X=(x_{ij})$ be an $m\times n$ matrix. Place the $mn$ entries of $X$ in decreasing (i.e., nonincreasing) order, to generate an $mn$-dimensional vector $X^*$. It is then said that $X$ majorises $Y$, written $Y\preceq X$, when $Y^*\preceq X^*$.
\end{defn}

\begin{defn}  Let $X$ be a matrix with entries in $S\subset \mathbb{R}$. $X$ is a minimal matrix if, for any other matrix $Y$, with entries in $S$ such that $Y\preceq X$, then $X^*=Y^*$.
\end{defn}

When the matrix entries represent probabilities of discrete bivariate distributions, then the majorisation ordering can be interpreted as an ordering of dependence, e.g., $X$ represents more ``dependence'' than $Y$ if $Y\preceq X$. Similarly to the discrete 
majorisation presented in Section \ref{sec:discrete}, condition (A2) holds for matrix majorisation.
From Joe \cite[Theorem 7]{Joe1985}, $X$ is a minimal matrix if it maximises $h(\cdot)$ subject to maintaining the row and column sums. We demonstrate that the minimal matrix corresponds to the bivariate distribution of two independent random variables for specific types of entropy.

\begin{example}\label{ex:entropy}
Let $X_1$ and $X_2$ be two independent random variables with $P(X_1=0)=\alpha$ and $P(X_2=0)=\beta$. We compute the joint probabilities: 
\begin{align}
p_{00} = \alpha \beta,\; p_{10} = (1-\alpha)\beta,\; p_{01} = \alpha (1-\beta),\; p_{11} = (1-\alpha)(1-\beta).
\end{align}
We can generate all binary distributions with the same margins as $X_1,X_2$ with a perturbation $\epsilon$:
\begin{align}
p_{00} = \alpha \beta + \epsilon,\; p_{10} = (1-\alpha)\beta-\epsilon,\; p_{01} = \alpha (1-\beta)-\epsilon,\; p_{11} = (1-\alpha)(1-\beta)+\epsilon,
\end{align}
with the restriction that $|\epsilon| < \min(p_{00}, p_{10}, p_{01}, p_{11}).$ If $\epsilon=0$, we retain the independence case. We compute the Shannon entropy, $ H(X_1, X_2)=-\sum_{i, j}p_{ij}\log(p_{ij})$, and by the maximum entropy principle, we have
\begin{align}
\frac{\partial H(X_1, X_2)}{\partial \epsilon}=\log\Bigg(\frac{\big(\alpha (1-\beta)-\epsilon\big)\big((1-\alpha)\beta-\epsilon\big)}{(\alpha\beta+\epsilon)\big((1-\alpha)(1-\beta)+\epsilon\big)} \Bigg).
\end{align}
Setting $\frac{\partial H(X_1, X_2)}{\partial \epsilon}=0$, we find $\epsilon=0$ and conclude that the Shannon entropy is at its maximum in the independence case.
We also compute the Tsallis entropy, $ H(X_1, X_2)=\sum_{i, j=}p_{ij}(1-p_{ij})$ with $\gamma=1$, 
\begin{align}
    H(X_1, X_2)    &=1-(\alpha\beta+\epsilon)^2-\big((1-\alpha)\beta-\epsilon \big)^2-\big(\alpha(1-\beta)-\epsilon \big)^2-\big((1-\alpha)(1-\beta)+\epsilon \big)^2.
\end{align}
As before we follow the maximum entropy principle to derive that the Tsallis entropy with $\gamma=1$ is at its maximum when $\epsilon = - \frac{1}{4}(2\beta-1)(2\alpha-1)$. We note that  $\epsilon$ is zero, the independence case, if at least one of $\alpha$ or $\beta$ is $\frac{1}{2}$. We observe that the maximum value of the Tsallis entropy could be obtained for $\epsilon\neq 0$ (dependence case). We conclude that the  independence case cannot be uniformly dominated within the ordering $\preceq$. 
\end{example}

Similar results hold for continuous multivariate densities. For two distributions with pdfs $f_1$ and $f_2$, where $f_1\preceq f_2$, this implies $f_2$ represents more ``dependence'' than $f_1$. In addition, the relative entropy function can be used to measure the dependence of distribution.

Joe \cite{Joe1987} introduced the concept of dependence parameters to construct the ordering of multivariate distributions. Let $a(\theta)$ represent a dependence parameter for a family of densities $f_{\theta}$ and $f_{\theta}\preceq f_{\theta'}$, if $a(\theta)\leq a(\theta')$ (or $f_{\theta}\preceq f_{\theta'}$, if $a(\theta')\leq a(\theta)$). For example, for zero-mean multivariate normal densities $f_{\Sigma_1}$ and $f_{\Sigma_2}$ parameterised by variance-covariance matrices $\Sigma_1$ and $\Sigma_2$ respectively, we have $f_{\Sigma_1}\preceq f_{\Sigma_2}$, if $|\Sigma_1|>|\Sigma_2|$. We are interested in demonstrating that the ordering imposed by the dependence parameters holds by using a DR introduced in Section \ref{sec:multivariate}. We derive the $DR$, $\tilde{f}(z)$, for $X\sim\text{MVN}(\bb{0}, \Sigma)$ with $\Sigma=\text{diag}\{\sigma^2, \dots, \sigma^2 \}$, that is
\begin{equation}
    \tilde{f}(z)=\frac{1}{(2\pi)^{n/2}|\Sigma|^{1/2}}\exp\Bigg\{-\frac{1}{2\sigma^2}\Big(\frac{z}{V_n} \Big)^{2/n} \Bigg\}.
\end{equation}

\begin{example} \label{example:binorm}
For bivariate normal densities $X\sim\text{MVN}(\bb{0}, \Sigma_1)$ and $Y\sim\text{MVN}(\bb{0}, \Sigma_2)$ with $\Sigma_1=\text{diag}(1, 1)$ and $\Sigma_2=\text{diag}(3, 3)$, we find DR cdfs
\begin{align}
\tilde{F}_1(z)=1-\exp\Big\{-\frac{z}{2\pi} \Big\}, \quad \tilde{F}_2(z)=1-\exp\Big\{-\frac{z}{6\pi} \Big\}.
\end{align}
Figure \ref{fig:PlotComp} shows that $\tilde{F}_2(z)\preceq \tilde{F}_1(z)$, since $|\Sigma_2|>|\Sigma_1|$, which is in line with results shown by Joe \cite{Joe1987}.

\begin{figure}[h!]
\begin{center}
\includegraphics[width = 0.37\textwidth]{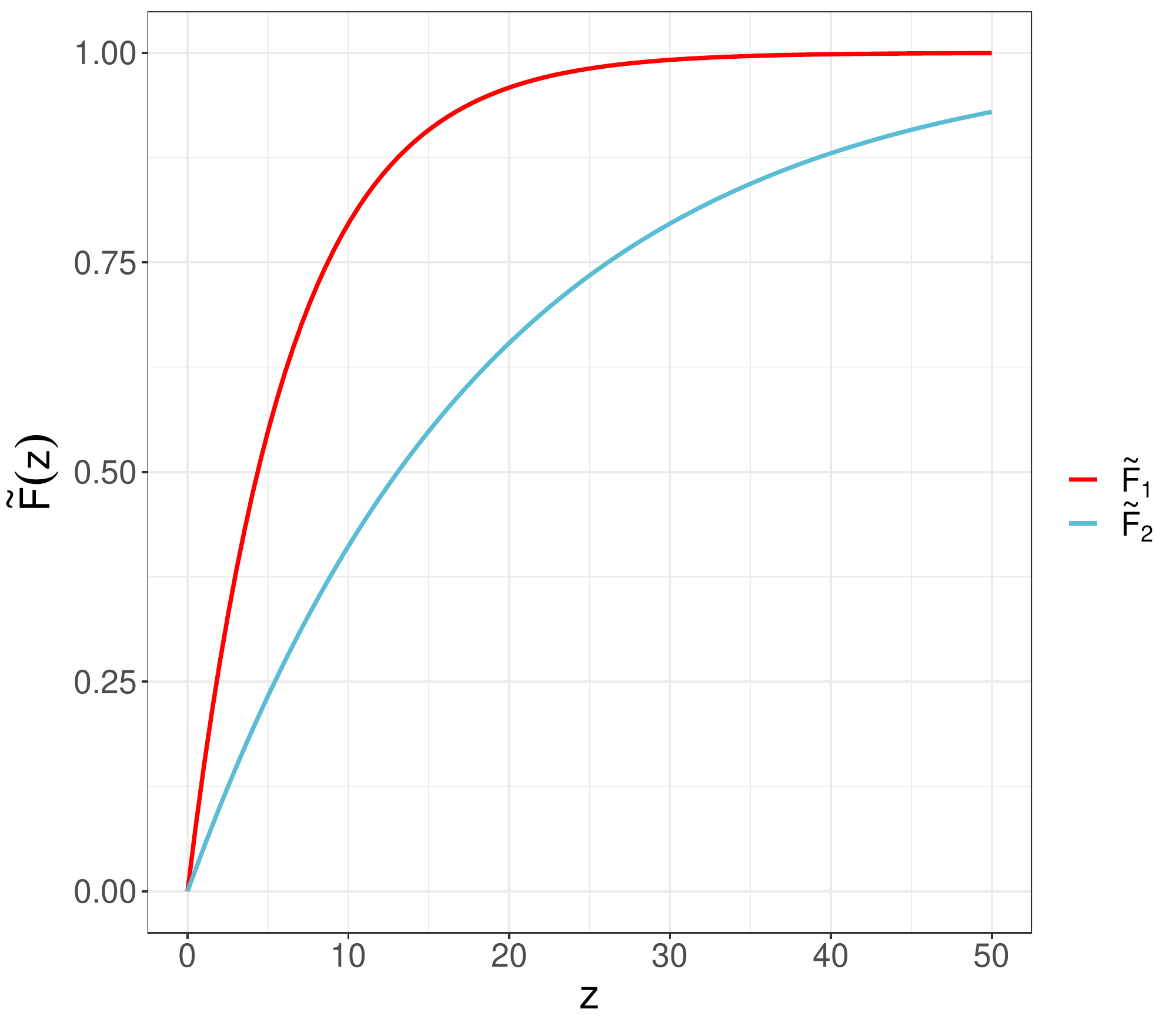}
\caption{Example \ref{example:binorm}. DR cdfs  $\tilde{F}_1(z)$ and $\tilde{F}_2(z)$.}\label{fig:PlotComp}
\end{center} 
\end{figure}
\end{example}

To obtain the orderings for multivariate normal distributions with non-zero off-diagonal entries in variance-covariance matrices, we could use empirical DRs discussed in Section \ref{sec:empirical}.

\subsection{Volume-contractive mappings}
\label{subsec:volume}
The area of the support is a key component of studying $\preceq$. For example, in the discrete case, if  $p= (p_1, p_2, p_3)$ are our probabilities with $p_1+p_2+p_3 = 1$, we have support of size 3. Splitting $p_3$ to form $q = (p_1, p_2, \frac{p_3}{2}, \frac{p_3}{2})$ has support of size 4 and we can conclude that $q \preceq p$. In the continuous case, we may refer to such an operation as dilation: locally, we have the same amount of density but with stretched support. Dilation in the continuous case is obtained via a transformation of the random variable, whose inverse we can call contractive. We proceed to demonstrate that the volume-contractive mapping implies a decrease in uncertainty.  

\begin{defn}
A differentiable and invertible function $ h:  \mathbb R \rightarrow \mathbb R$, $y = h(x)$, is a volume-contractive mapping if the absolute value of the Jacobian determinant: $J = \vline \frac{\partial y}{\partial x} \vline$ satisfies $ 0 < J \leq 1$ for $x \in \mathbb R$.
\end{defn}

\begin{lemma}
If $h:\mathbb R \rightarrow \mathbb R$ is a volume-contractive mapping, then, for any random variable $X \sim f_X(x)$, it holds that
$X \preceq Y = h(X)$.  
\end{lemma}
\begin{proof} We give a proof for the one dimensional case only and, in addition, assume $f_X(x)$ and $f_Y(y)$ are invertible. Using the slice condition, we want to show that
$$\mbox{prob}_X \{ f_X(X) \geq c\} \geq  \mbox{prob}_Y\{ f_Y(Y) \geq c\}. $$
Developing the left hand side, we see that
\begin{eqnarray*}
\{ f_X (X) \geq c \}  \Leftrightarrow  \{X \geq f_X^{(-1)}(c) \} 
                               \Leftrightarrow  \{Y \geq h(f_X^{(-1)}(c)) \}.
\end{eqnarray*}

Computing the density of $Y$ as
$$ f_Y(y) = \frac{1}{J} f_X(h^{(-1)} (y)),$$
gives
$$ f_Y^{(-1)}(c) =   h\left(f_X^{(-1)} (Jc)\right).$$
We thus need to establish whether
$h(f_X^{(-1)} (c)) \geq h( f_X^{(-1)} (Jc)).$
We see the statement reduces to
$c \geq J c,$ which holds by assumption.
\end{proof}

\section{Algebra for uncertainty}\label{sec:algebra}

\begin{defn}\label{def:invCDF}
For any two DR pdfs $\tilde{f}_1(z)$ and $\tilde{f}_2(z)$, define $\tilde{F}_1(z) \otimes \tilde{F}_2(z)$, the associated cdf of the density
\begin{align} \label{eq:x}
\tilde{f}_1(z) \; [+]_{ \frac{1}{2}} \; \tilde{f}_2(z) =  \left(\tilde{f}_1^{(-1)}(2z) + \tilde{f}_2^{(-1)}(2z) \right)^{(-1)}.
 \end{align}
\end{defn}

\begin{defn}\label{def:lattice}
For any two DR cdfs $\tilde{F}_1$  and $\tilde{F}_2 (z)$, define
\begin{equation}
\tilde{F}_1(z) \vee \tilde{F}_2 (z)  = \max(\tilde{F}_1(z), \tilde{F}_2(z)),\end{equation}
\begin{equation}
    \tilde{F}_1(z) \wedge \tilde{F}_2 (z)  = \min(\tilde{F}_1(z), \tilde{F}_2(z)),
\end{equation}
which themselves are cdfs.
\end{defn}

The partial ordering $\preceq$ under the meet and join $\vee$ and $\wedge$ defines a lattice
which we refer to as the {\em uncertainty lattice}. It is satisfying that the `meet' and `join' which are defined once $\preceq$ is established can be manifested by the max and min of Definition \ref{def:lattice}. Since we can embed a multidimensional distribution with density $f(x)$ into the one-dimensional DR, we claim that the lattice is universal.

The inverse mixing cdf $\otimes$  can be combined with $\vee$ (or $\wedge$). We modify the notation and replace $\vee$ by $\oplus$, which implies that $\oplus$ and $\otimes$ yield a so-called max-plus algebra (also called tropical algebra)  \cite{maclagan2015introduction}.  For this to be valid, we need to demonstrate that the distributive property holds.
 
\begin{lem}\label{dist}
We have
\begin{equation}
 \tilde{F}_3(z) \otimes ( \tilde{F}_1(z) \oplus \tilde{F}_2(z)) =  (\tilde{F}_3(z) \otimes \tilde{F}_1(z)) \oplus  (\tilde{F}_3(z) \otimes \tilde{F}_2(z)),
\end{equation}
where $\tilde{F}_1(z),\tilde{F}_2(z)$ and $\tilde{F}_3(z)$ are DR cdfs and  $\tilde{F}_1(z) \oplus \tilde{F}_2 (z) = \max(\tilde{F}_1(z),\tilde{F}_2(z))$.
\end{lem}

\begin{proof}  
 The proof follows by switching the min and max when we take the inverses:
\begin{align*}
 \tilde{F}_3(z) \otimes \left( \tilde{F}_1(z) \oplus \tilde{F}_2(z) \right) % & = \left (\tilde{F}_3^{(-1)}(2z)  + \left (  \tilde{F}_1(2z) \oplus  \tilde{F}_2(2z) \right )^{(-1)}\right)^{(-1)},\\
  & = \tilde{F}_3^{(-1)}(2z)  + \left (\max \left (  \tilde{F}_1(2z) ,  \tilde{F}_2(2z) \right )^{(-1)}\right)^{(-1)},\\
    & =  \left (\tilde{F}_3^{(-1)}(2z)  +\min \left (  \tilde{F}^{(-1)}_1(2z) ,  \tilde{F}^{(-1)}_2(2z) \right ) \right)^{(-1)},\\
        & =  \left ( \min \left(
        \tilde{F}_3^{(-1)}(2z)  +  \tilde{F}^{(-1)}_1(2z),
         \tilde{F}_3^{(-1)}(2z)  +  \tilde{F}^{(-1)}_2(2z) \right ) \right)^{(-1)},\\
             & =   \max \left(
        \left (\tilde{F}_3^{(-1)}(2z)  +  \tilde{F}^{(-1)}_1(2z)\right )^{(-1)}, \left (         \tilde{F}_3^{(-1)}(2z)  +  \tilde{F}^{(-1)}_2(2z) \right )^{(-1)}\right ) ,\\
         & =   \max \left ( 
        \tilde{F}_3(z) \otimes \tilde{F}_1(z)), \tilde{F}_3(z) \otimes \tilde{F}_2(z) \right),
\end{align*}
from which the result is obtained.
 
 \end{proof}

\begin{defn}
The \emph{uncertainty ring} is the toric (semi) ring of non-decreasing, twice-differentiable functions on $[0, \infty)$  on  $\otimes$ and $\oplus$ and with $\oplus$ identity as $-\infty$.
\end{defn}
 We note that the $\otimes$ unit element will be $0$ and the $\oplus$ unit element will be $-\infty$. To obtain proper decreasing densities, we need to impose the additional condition that $\tilde{f}(z)$ is decreasing, $\tilde{F}(z)$ is a non-negative function, $\tilde{F}(0) = 0$ and $\tilde{F}(z) \rightarrow 1$ as $z \rightarrow \infty$. 
To introduce polynomials which comprise the ring needs the concept of a power. Consider the pdf arising from $\tilde{F}_1(z) \otimes \tilde{F}_1(z)$,
\begin{align}
\tilde{f}(z) = \left( \tilde{f}_1^{(-1)}(2z) + \tilde{f}_1^{(-1)}(2z) \right)^{(-1)} = 
                \frac{1}{2} \tilde{f}_1\left(\frac{z}{2} \right), 
\end{align} 
which is the pdf for the scaled random variable $Y= 2 X$ where $X \sim \tilde{f}_1(z)$, and 
the cdf for $Y$ is $\tilde{F}_1\left(\frac{z}{2}\right)$. In general, we define the $k^{\text{th}}$ $\otimes$ power as
\begin{align}
\otimes^n \tilde{F} (z) = \tilde{F}\left( \frac{z}{n} \right).
\end{align}
The intuition of this expression is that increasing powers represent increasing dilation and form a decreasing chain with respect to our $\preceq$ ordering. A monomial with respect to $\otimes$ takes the form
\begin{align}
\prod_{i=1}^{m} \otimes^{\alpha_i} \tilde{F}_i(z). 
\end{align}
Adjoining the base field $\mathbb R$, and appealing to Lemma \ref{dist}, we can define a ring of tropical polynomials \cite{glicksberg1959convolution}.

We summarise the operations that we have:
\begin{enumerate}
\item Scalar multiplication $\tilde{F}(z) \rightarrow \beta\tilde{F}(z)$, $ \beta \in \mathbb R$.
\item Inverse mixing $\tilde{F}_1(z) \otimes \tilde{F}_2(z)$ and $\otimes$ powers and monomials/polynomials
\item Maximum and minimum of $\tilde{F}_1(z)$ and $\tilde{F}_2(z)$, denoted $\vee$ and $\wedge$, respectively.
Noting that $\vee$ is written $\oplus$, when discussing the ring. We can also define a min-plus algebra and may use $\oplus$.
\item The one-dimensional DR from independent pairs of random variable $(\tilde{F}_1,\tilde{F}_2)$, where $\tilde{F}_1$ and $\tilde{F}_2$ can come from different distributions.
\item Convolutions $\tilde{F}_1(z) * \tilde{F}_2(z)$. This refers to the DR cdf of the sum of independent random variables $X_1 \sim f_1(x)$ and $X_2 \sim f_2(x)$. 
\end{enumerate}

Further natural developments using ring concepts such as ideals are the subject of further work.  In fact, convolutions themselves form a semi-group, but we do not delve into the relationship between our ring and that semi-group. It is instructive to work over the binary field so that we do not have to use full scales from $\mathbb R$, but only $\{0,1\}$. This also has the advantage that, in every polynomial, we have proper pdfs and cdfs. An analogy is Boolean algebra. In the next example we illustrate these operations to demonstrate the complexity that can arise from from a single distribution.

\begin{example} \label{sec6:example_exp} 
Let $X_1 \sim \exp\{-x_1\}, $ and $X_2\sim \exp\{-(x_1+x_2) \}$ with $x_1, x_2>0$. The DRs are given by Equation (\ref{eq:DRM_mult_exp}), from which the DR cdfs are
\begin{eqnarray}
\tilde{F}_1(z) & = &   1-e^{-z}, \\
\tilde{F}_2(z) & = &  1-(1+\sqrt{2z})e^{-\sqrt{2z}}.
\end{eqnarray}

To compute $\tilde{F}_3(z)  =  \otimes^2 \tilde{F}_1(z)$, we first calculate the inverse mixing of $\tilde{f}_1(z)$ with itself and $\alpha=1/2$, 
\begin{align}
      \tilde{f}_3(z)     = \big(-2\log(2z) \big)^{(-1)}=\frac{1}{2}e^{-z/2},
\end{align}
with the corresponding cdf,
\begin{equation}
  \tilde{F}_3(z)=1-e^{-\frac{z}{2}}.  
\end{equation}

Similarly, to compute $\tilde{F}_4(z)=\otimes^2 \tilde{F}_2(z)$, we first obtain the two solutions for $\tilde{f}_4(z)=\exp\{-\sqrt{z}/2\}$ and $\tilde{f}_4(z)=\exp\{\sqrt{z}/2\}$. As the second solution does not satisfy the definition of DRs, we derive the cdf from the first:
\begin{equation}
    \tilde{F}_4(z) = 1-(1+\sqrt{z})e^{-\sqrt{z}}.
\end{equation}
Note that the following relationships between DR cdfs hold:
$    \tilde{F}_3(z) = \tilde{F}_1 (z/2 ) $ and $ \tilde{F}_4(z) = \tilde{F}_2(z/2 )$. Finally, we compute the DR cdf for the convolution of two univariate standard exponential random variables, with pdf $f_3(x) = x\exp\{-x\}$, i.e., $X_3=X_1+X_1$, denoted as $\tilde{F}_5(z) =  (\tilde{F}_1(z) * \tilde{F}_1(z))$, we employ the slice method introduced in Example \ref{example|_beta}. We have
\begin{equation}
   \tilde{F}_5(z) = \exp\Big\{-\frac{z}{e^z-1}\Big\} - \exp\Big\{-\frac{ze^z}{e^z-1} \Big\}. 
\end{equation}

\begin{figure}[h!]
\begin{center}
\includegraphics[width = 0.4\textwidth]{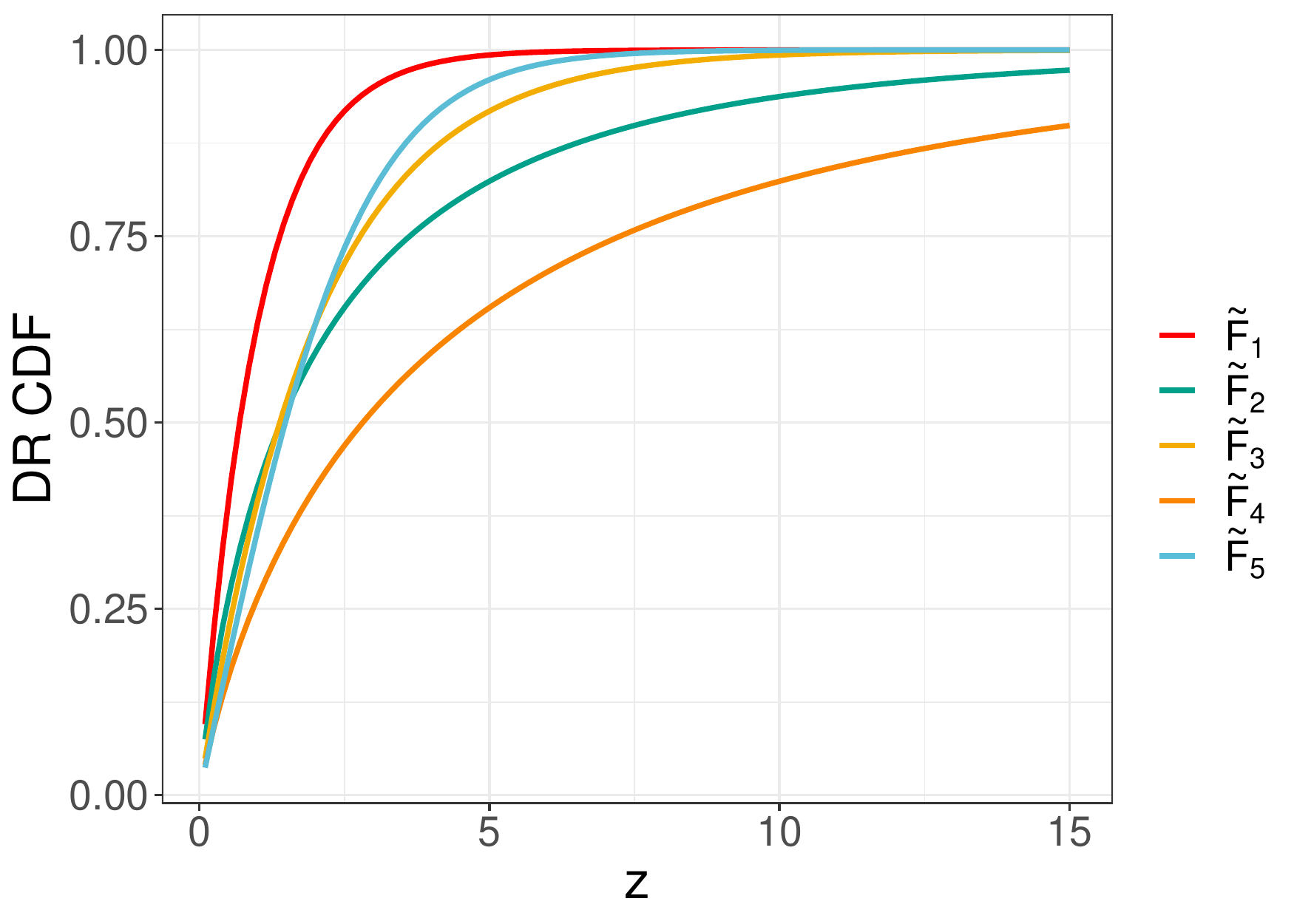}
\caption{Example \ref{sec6:example_exp}. DR cdfs for various operations on $X_1$ and $X_2$.}
\label{fig:PlotExample10}
\end{center}
\end{figure}

Figure \ref{fig:PlotExample10} illustrates the relationships between DR cdfs and we observe that 
$    \tilde{F}_4(z)\preceq\tilde{F}_2(z)\preceq \tilde{F}_1(z),$    $ \tilde{F}_4(z) \preceq\tilde{F}_3(z)\preceq \tilde{F}_1(z),$ and 
$      \tilde{F}_4(z)\preceq\tilde{F}_5(z)\preceq \tilde{F}_1(z)$.
and there are no orderings between $\tilde{F}_2(z), \tilde{F}_3(z)$ and $\tilde{F}_5(z)$. Figure \ref{fig:PlotExample10_2} shows that under $\vee$ and $\wedge$, we have the following sets of inequalities
 \begin{align}\label{lattice_eqn}
 \begin{array}{cc}
   & \tilde{F}_4(z)\preceq \tilde{F}_2(z)\wedge \tilde{F}_3(z)\preceq \tilde{F}_2(z)\vee \tilde{F}_3(z) \preceq \tilde{F}_1(z),\\
   & \tilde{F}_4(z)\preceq \tilde{F}_3(z)\wedge \tilde{F}_5(z)\preceq \tilde{F}_3(z)\vee \tilde{F}_5(z) \preceq \tilde{F}_1(z),\\
  &  \tilde{F}_4(z)\preceq \tilde{F}_2(z)\wedge \tilde{F}_5(z)\preceq \tilde{F}_2(z)\vee \tilde{F}_5(z) \preceq \tilde{F}_1(z),
 \end{array}
\end{align}
from which the full lattice can be formed.

\begin{figure}[h!]
\begin{center}
\includegraphics[width = 1\textwidth]{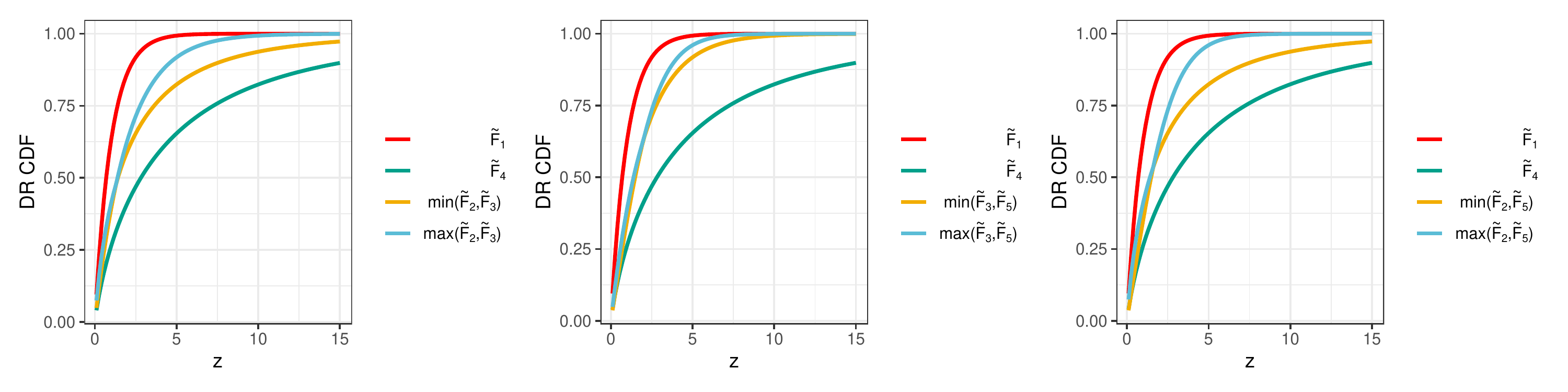}
\caption{Example \ref{sec6:example_exp}. DR cdfs of Equation \eqref{lattice_eqn}.}
\label{fig:PlotExample10_2}
\end{center}
\end{figure}
\end{example}

Given the equivalence relation (B2) in Equation \eqref{equiv_convex}, we can study how the above structures affect the manipulation of uncertainty measured by a single metric. Since $H(f(x)) = \int_{0}^{\infty} h(\tilde{f}(z)) dz$ , without loss of generality, we can consider $H$ as a functional of the DR with the advantage that the operations $\otimes,\oplus$ can be applied. If we are able either theoretically or computationally to show that the ordering holds, then we feel that this strong condition deserves to be a candidate for a universal version of what we may mean by ``more certain'' or ``more uncertain''. We have seen with the above examples that computing whether two distributions can be compared according to $\preceq$ may be not be an easy computation.  This relies on the difference of the DR cdfs not having a zero value on $[0,\infty)$. What we term the ``ring" above concerns equalities, not inequalities. The relationship between these equalities and the partial ordering $\preceq$ requires a fuller development.  With $\vee$ and $\wedge$, the situation is clearer, but with other operations, this is not the case and may depend radically on the distributional family concerned. We present a practical situation in which the results above can be used.

\begin{example}
Consider $2$ horseraces each with a different number of horses $n_i$ ranked by a punter (or bookmaker) in order of their probability of winning:  
$p_{(i,1)} \geq \cdots \geq p_{(i,n_i)}$, such that $\sum_{j=1}^{n_i}p_{i,j}=1$. If the two sets of horses are to be combined into a single race, the issue then is how to combine the probabilities. Dividing each probability by two, combining and ranking them, i.e, $\tilde{F}_1(z) \otimes \tilde{F}_2(z)$, is inverse mixing with $\alpha = \frac{1}{2}$. One can imagine some effect which may lead to having $\alpha \neq \frac{1}{2}$ such as the track being wet and not suited to one set of horses. 

Consider the case of a single race with two punters that rank the horses in the same order, but with different probabilities: $p_{(1,1)}\geq \ldots \geq p_{(1,n_1)}$ and  $q_{(1,1)} \geq \ldots \geq \ldots q_{(1,n_1)}$. To define a joint betting strategy, a set of odds combining each of the individuals' odds could take different approaches: an optimistic, more certain, approach with $\tilde{F}_1(z) \vee \tilde{F}_2(z)$ or a more pessimistic, uncertain, approach with $ \tilde{F}_1(z) \wedge \tilde{F}_2(z)$. This argument is predicated on the rank order being the same, otherwise the same horse may appear twice in the min or max ordering. The min or max may then refer to a kind of hypothetical race. Nonetheless, we suggest that they are useful notionally. The same issue arises if one considers an average of the two actual probabilities, direct mixing, 
$\frac{1}{2}(p_{(1,1)} +q_{(1,1)} )\geq \ldots \geq\frac{1}{2}(p_{(1,n_1)}+q_{(1, n_1)}),$ which corresponds to taking $\frac{1}{2}(\tilde{F}_1(z) + \tilde{F}_2(z))$.
 \end{example}

\section{Empirical decreasing rearrangements}\label{sec:empirical}
We present two approaches for deriving the empirical DR and its associated cdf for the analysis of an experimental data set. In Section \ref{subsec:Discrete_Climate}, we assess the uncertainties associated with climate projections in two dimensions. In Section \ref{subsec:Cont_Heat}, we present two algorithms to obtain approximations for $\tilde{f}(z)$ and $\tilde{F}(z)$ for data sets in higher dimensions, and apply this to energy systems planning in Section \ref{subsubsec:DHE}.

\subsection{Climate projections}\label{subsec:Discrete_Climate}
The 2018 UK climate projections \cite{UKCP18} considered four different scenarios, called Representative Concentration Pathways (RCP), for greenhouse gas concentrations in the atmosphere.
These scenarios contain a range of inputs that reflect socio-economic change, technological change, energy use and emissions of greenhouse gases and air pollutants, which are used to study the impact on climate through to the year 2100.  We consider two variables: (i) the increase in mean air temperature at a height of 1.5 m, and (ii) the percentage increase in precipitation, where each variable is relative to the baseline period of 1981-2010.
The projections illustrated in Figure~\ref{fig:climate_scatter} correspond to mean daily values over the period 2050-2079. The data is discretised into twenty categories, as temperature anomaly is divided into five categories and precipitation into four. From this, the probabilities are ordered, and we obtain empirical DR cdfs in the left panel of Figure \ref{fig:climate_example_pdf_cdf2}. Observing that the cdf for RCP2.6 majorises that of RCP8.5, we conclude that RCP8.5 is more uncertain than RCP2.6. The maximum and minimum of the cdfs are given in the right panel, and RCP2.6 carries the lowest level of uncertainty among the considered scenarios since its cdf corresponds to $\tilde{F}_1(z) \lor\tilde{F}_2(z)\lor\tilde{F}_3(z)\lor\tilde{F}_4(z)$, where the subscript indicates the scenario. In contrast, RCP8.5 carries the most uncertainty, as its cdf corresponds to $\tilde{F}_1(z)	\land\tilde{F}_2(z)\land\tilde{F}_3(z)\land\tilde{F}_4(z)$. 
In this analysis, majorisation identifies one comparison of universal uncertainty, illustrating that scenarios are associated with different levels of uncertainty. In addition, transforming the climate projection data into DR cdfs allows us to visualise the uncertainty associated with RCP scenarios and is therefore a tool for the communication of uncertainty in settings with decision makers and stakeholders \cite{gov_toolkit}.

\begin{figure}[ht]
\begin{center}
\includegraphics[width=0.55\textwidth]{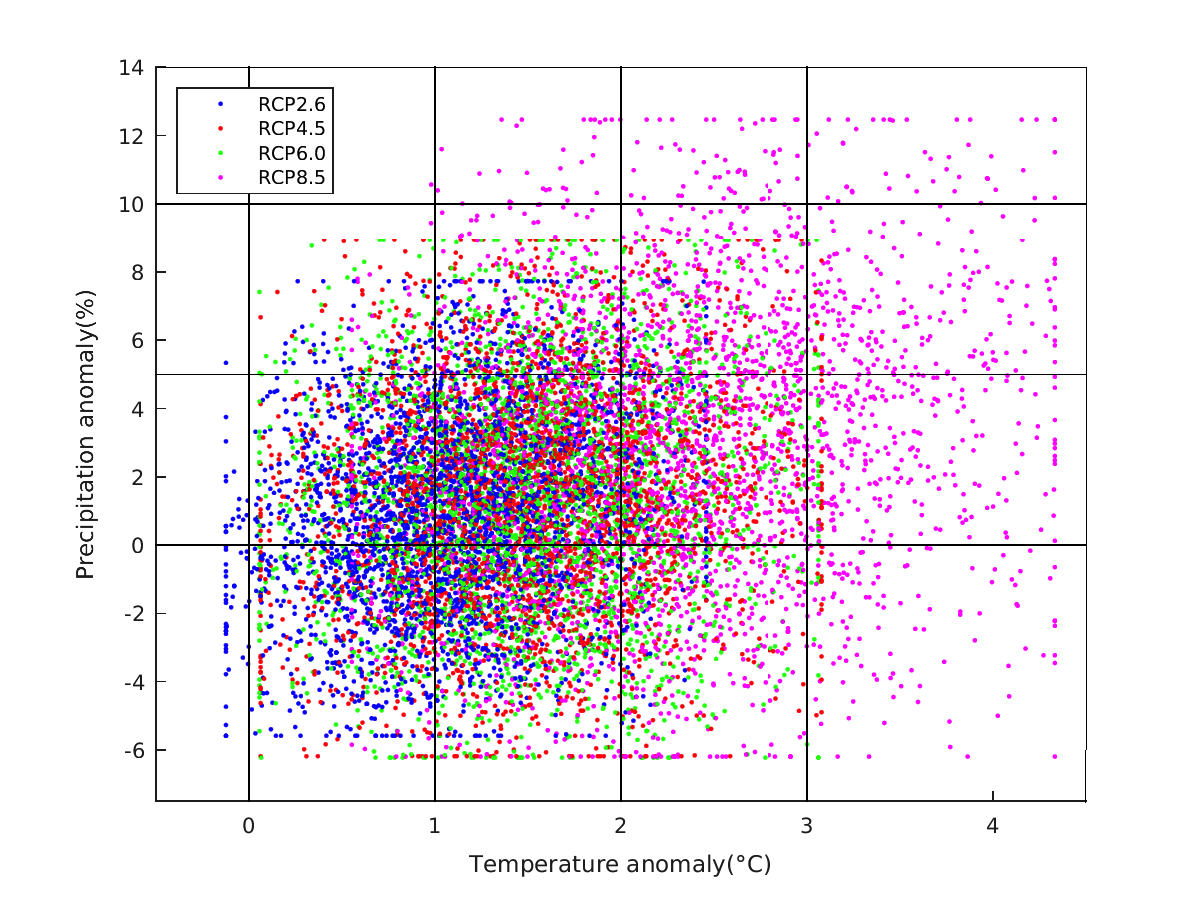}
\vspace{-1em}
\end{center}
\caption{Projections of mean daily values over the period 2050-2079 \cite{UKCP18}. Each point represents an ensemble member and each colour represents a different RCP.}
\label{fig:climate_scatter}
\end{figure}

\begin{figure}[h]
\begin{center}
\includegraphics[width=0.75\textwidth]{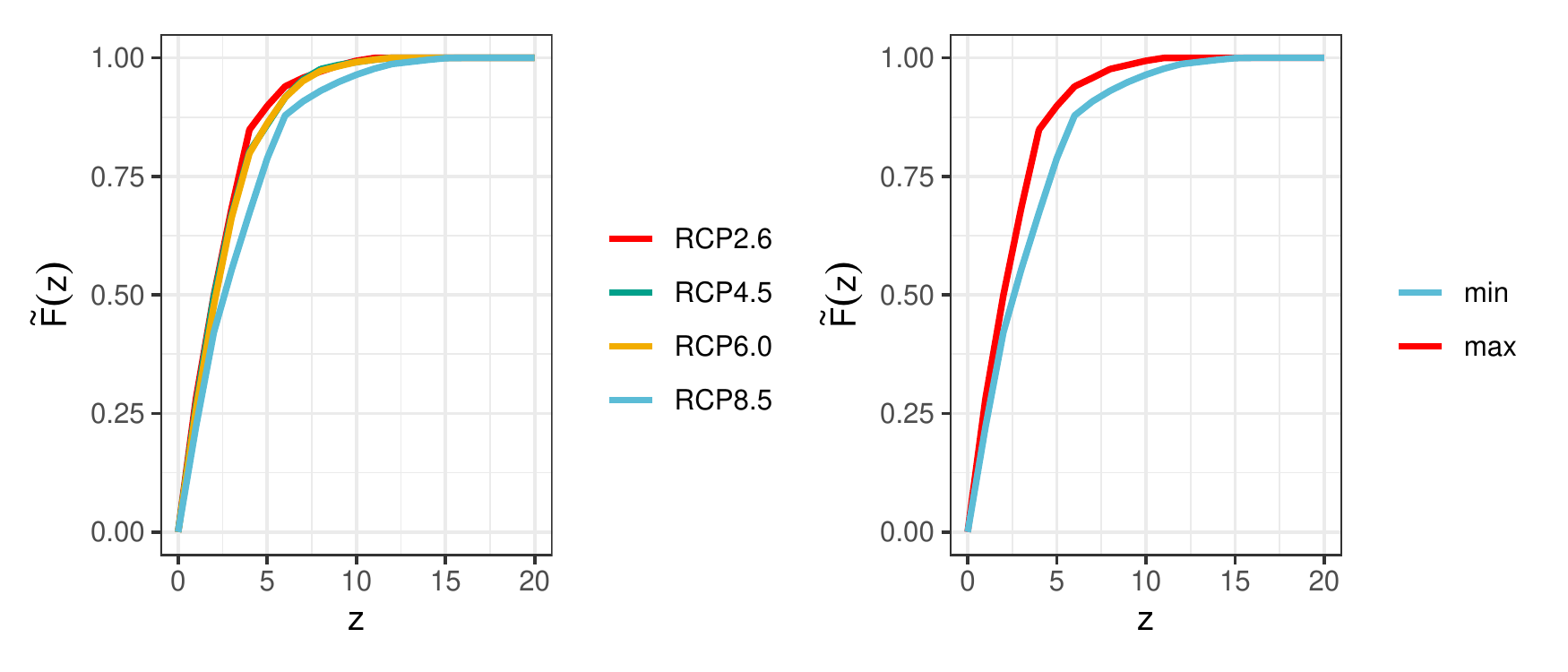}
\vspace{-1em}
\end{center}
\caption{\textit{Left panel}: Empirical DR cdf for each RCP scenario. \textit{Right panel}: the representation of $\tilde{F}_1(z) \lor\tilde{F}_2(z)\lor\tilde{F}_3(z)\lor\tilde{F}_4(z)$ and $\tilde{F}_1(z)	\land\tilde{F}_2(z)\land\tilde{F}_3(z)\land\tilde{F}_4(z)$.}
\label{fig:climate_example_pdf_cdf2}
\end{figure}

\subsection{Majorisation in higher dimensions}\label{subsec:Cont_Heat}
Consider the data set $x_{ij}$ of data points $ i=1, \dots, m$ and dimension $j=1, \dots, n$. To obtain the DR, we require the density function values to construct the measure (distribution) function $m(y)$. Assume the observed data is sampled from a population with unknown pdf $f_X(x_1, \dots, x_n)$, from which we estimate the pdf $\hat{f}_X(x_1, \dots, x_n)$. We employ kernel density estimation (KDE) \cite{Parzen1962} to obtain $\hat{f}_X(\cdot)$ using the \texttt{ks} package in \texttt{R}, which automatically selects the bandwidth parameters \citep{ks2020}. Alternative approaches such as density forests \cite{Criminisi2013} and the $k$-nearest neighbour density estimation algorithm \cite{Bishop2006} would also be appropriate. To obtain empirical DRs, we adopt a two stage process for $\tilde{f}_{\hat{f}}(z)$ as described in Algorithm 1. The first stage involves obtaining the distribution function $m(y)$, which is used in the second stage to to derive the DR.
 
\begin{algorithm}[H]
\SetAlgoLined
Based on data $x_{ij}\in R, i=1, \dots, m$ and $j=1, \dots, n$, fit a pdf $\hat{f}_X(x_1, \dots, x_n)$ using KDE\;
Produce a uniform and/or space-filling set $S$ of size $N$ across the input space $R$, with $s\in S$\;
 \For{$y=y_1, \dots, y_M$}{
  Derive a set $S_y=\big\{s\in S:\hat{f}_X(s)>y \big\}$ of size $N_y=\vert S_y\vert$\;
  Estimate the volume of $S_y$, i.e., $m_{\hat{f}}(y)=\text{Vol}(S_y)$ by the Monte Carlo method\;
 }
 Plot the estimated measure function values, $m_{\hat{f}}(y)$ against $y$\;
 Swap the axes, so that $\tilde{f}_{\hat{f}}(z)$ and $z$ correspond to $y$ and $m_{\hat{f}}(y)$.
 \caption{Empirical DR $\tilde{f}_{\hat{f}}(z)$.}
\end{algorithm}

Monte Carlo integration is used to estimate the volume of domain $S_y$ to derive the measure function $m_{\hat{f}}(y)$. In particular, \cite{Fok1989} proposed specifying another domain $R$ (a hypercube or a hyperplane) of known volume $\text{Vol}(R)$, such that $S_y\in R$. The ratio of two volumes, $p=\text{Vol}(S_y)/\text{Vol}(R)$, and the volume $\text{Vol}(S_y)$ are estimated by $\hat{p}=N_y/N$ and $\hat{\text{Vol}}(S_y)=\hat{p}\text{Vol}(R)$.

We demonstrate the use of Algorithm 1 by generating a random sample of size $m=200$ from the standard bivariate normal distribution, with DR given in Equation (\ref{eq:DRM_mult_normal}). To apply the algorithm, we produce a uniform sample of points of size $N=2500$ across the domain $R=[-5, 5]\times [-5, 5]$ of $\text{Vol}(R)=10^2$. In the left panel of Figure \ref{fig:EmpDRFinal1} we depict the estimated values of the distribution function $m_{\hat{f}}(y)$ against $y$ and
note that the smoothness of the estimated distribution  depends on $M$, thus we expect to obtain a smooth representation with large $M$ (cutoffs in density). In the right panel of Figure \ref{fig:EmpDRFinal1} we compare $\tilde{f}_{\hat{f}}(z)$ with $\tilde{f}(z)$ and observe that the empirical DR (red dashed line) overlaps with the DR (blue solid line).

\begin{figure}[h!]
\begin{center}
\includegraphics[width=.7\textwidth]{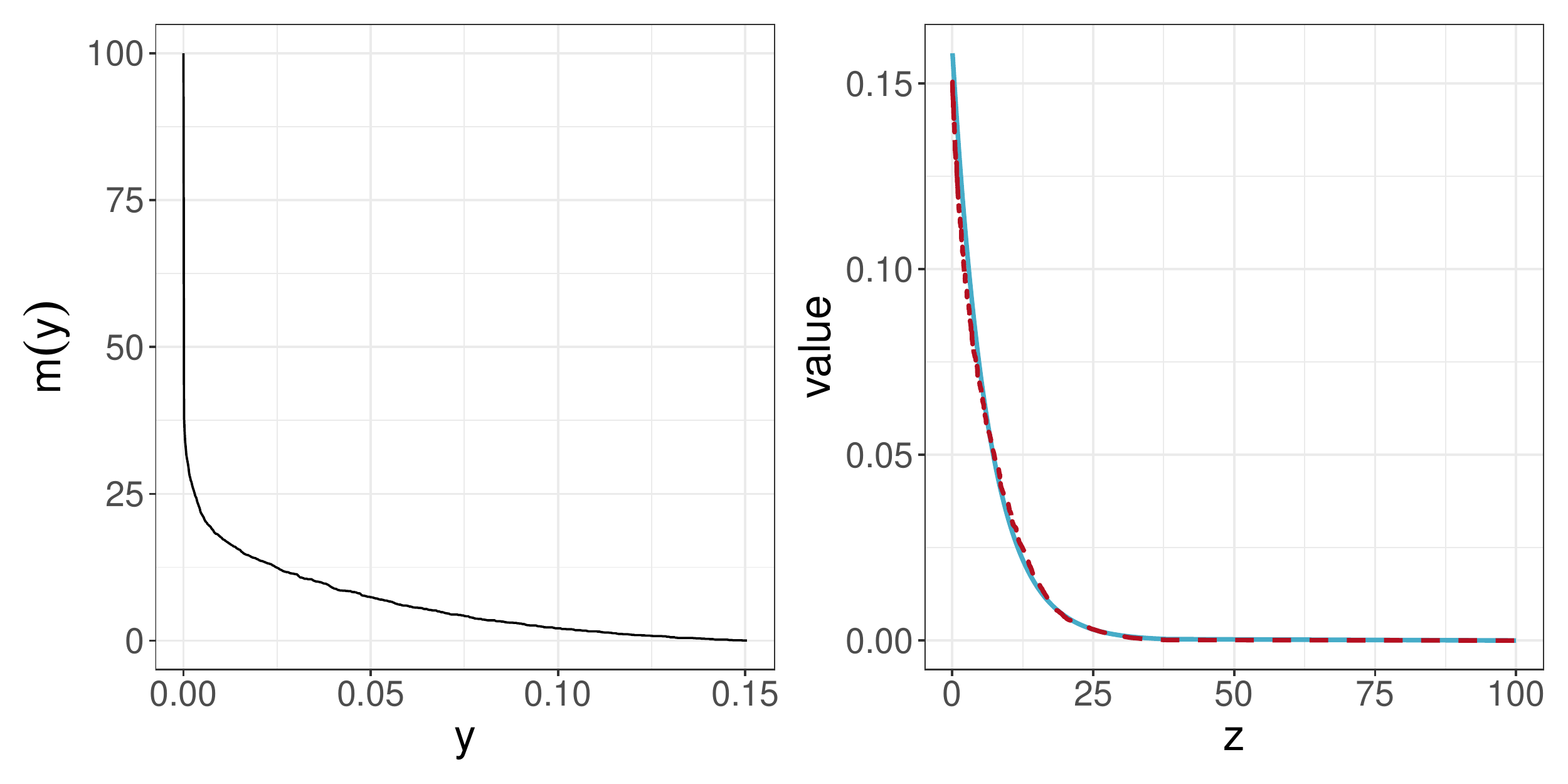}
\vspace{-1em}
\end{center}
\caption{\textit{Left panel}: Estimated measure function $m_{\hat{f}}(y)$. \textit{Right panel}: 
$\tilde{f}(z)$ (blue line)  and $\tilde{f}_{\hat{f}}(z)$ (red dashed line) .}
\label{fig:EmpDRFinal1}
\end{figure}

\begin{algorithm}[h]
\SetAlgoLined
Specify an equally spaced vector $\boldsymbol{z}^*=(z_1^*, z_2^*, \dots, z_l^*)$\;
Fit a linear interpolator (spline) through $\{z_i, \tilde{f}_{\hat{f}}(z_i)\}_{i=1}^M$ (these values were derived in Algorithm 1) to obtain values of $\tilde{f}_{\hat{f}}(z_i^*), i=1, \dots, l$\;
 \For{$z_i^*, i=1 \dots, l-1$}{
  Estimate probability values $P(z_i^*<z<z_{i+1}^*)$ by numerical integration\;
  Obtain values $\tilde{F}_{\hat{f}}(z_i^*)=\frac{\sum_{k=1}^{i-1}P(z_k<z<z_{k+1})}{\sum_{k=1}^{l-1}P(z_k<z<z_{k+1})}$\;
 }
 Plot $\tilde{F}_{\hat{f}}(z^*)$ against $z^*$.
 \caption{Empirical DR cdf $\tilde{F}_{\hat{f}}(z)$.}
\end{algorithm}

We present Algorithm 2 for obtaining an empirical cdf of the DR, an approximation to $\tilde{F}(z)$, denoted by $\tilde{F}_{\hat{f}}(z)$. The weighting of computed probabilities by 
$(\sum_{k=1}^{l-1}P(z_k<z<z_{k+1}))^{-1}$
comes from the assumption that $z$ is upper bounded and we can only compute probabilities at the values specified in $\boldsymbol{z}^*$. Therefore, we expect $\sum_{k=1}^{l-1}P(z_k<z<z_{k+1})=1$. However, we tend to observe this sum to be slightly less than one due to errors introduced by numerical integration. In Figure \ref{fig:CDFDR2D} we apply the algorithm on the bivariate data set used previously in this section. The closed form expression for $\tilde{F}(z)$ is
$ \tilde{F}(z)=1-\exp \{-\frac{z}{2\pi}  \}$.
From the right panel, it can be seen that the empirical cdf $\tilde{F}_{\hat{f}}(z)$ is an accurate representation of $\tilde{F}(z)$. 

\begin{figure}[h!]
\begin{center}
\includegraphics[width=.7\textwidth]{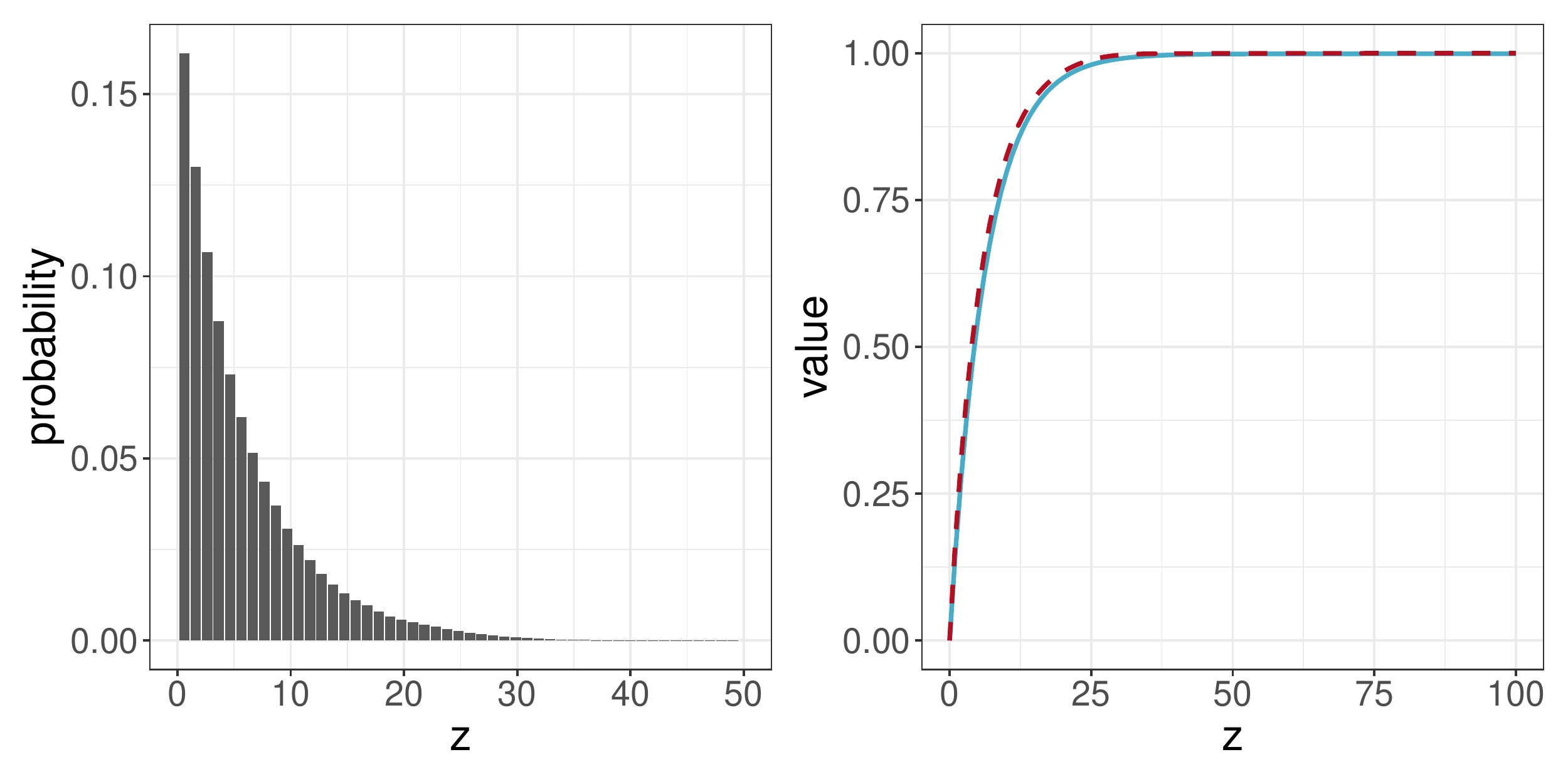}
\vspace{-1em}
\end{center}
\caption{\textit{Left panel}: binned probability representation of the empirical DR, $\tilde{f}_{\hat{f}}(z)$,  obtained as part of Algorithm 2. \textit{Right panel}: empirical DR cdf, $\tilde{F}_{\hat{f}}(z)$ obtained from Algorithm 2. The blue solid line and red dashed line correspond to $\tilde{F}(z)$ and $\tilde{F}_{\hat{f}}(z)$ respectively.}
\label{fig:CDFDR2D}
\end{figure}

\subsection{Energy systems planning}
\label{subsubsec:DHE}
We compare the uncertainty associated with three potential design options for supplying heat to a residential area in Brunswick, Germany, considered as part of EU project ReUseHeat (Recovering urban excess heat) \cite{REUSEHEAT}.  District heating networks allow heat from a centralised source to be distributed to buildings through a network of insulated pipes \cite{werner2013district}, and the primary objective of the project is to demonstrate the use of low temperature sources of heat in these networks. 

The city's existing district heating network is powered by a Combined Heat and Power (CHP) plant, which uses natural gas as a fuel and outputs both heat for use in the network and electricity. The network in the newly constructed area of interest will be connected to the CHP and, in addition, there is an option to use excess heat from a nearby data centre to provide at least some of the heat to the district. Excess heat from a data centre is a low temperature source which requires an electric heat pump to ``upgrade'' the temperature before being suitable for use in the system.
 
We are interested in the uncertainty ordering for three heating design options: (1) CHP, (2) CHP and Heat Pump, (3) Heat Pump; considering the two variables: Net Present Cost (NPC) and $\text{CO}_2$-equivalent emissions (in metric tonnes). Using an energy systems simulation (OSeMOSYS \cite{Howells2011}), we produce predicted outputs for these variables. We define three scenarios by varying a number of inputs to the simulations, in particular elements of government climate policy and consumer engagement with green technology.  These are shown in Table \ref{tab:Scenarios}. We refer to Volodina \emph{et al.} \cite{Volodina2020} for further details.

\begin{table}[h!]
\caption{District heating study scenarios \cite{Volodina2020}.}  
\begin{center}
\begin{small}
\begin{tabular}{c|lll}
\midrule
\textbf{Scenario} & \textbf{Emission Penalty} & \textbf{Consumer demand} & \textbf{Commodity prices}\\
Green & 100\euro/metric tonne & -1\% annual change &$\uparrow$ gas, $\downarrow$ electricity\\
Neutral & 40\euro/metric tonne &small fluctuations & small fluctuations\\
Market & no penalty & +1\% annual change & $\downarrow$ gas, $\uparrow$ electricity\\
\midrule
\end{tabular}
\end{small}
\end{center}
\label{tab:Scenarios}
\vspace{-1em}
\end{table}

Figure \ref{fig:ScatterPlot} shows the distribution of points produced for each design option and scenario. Design option 1 is associated with the highest level of emissions due to the use of natural gas, while design option 3 has the lowest emissions levels but has the highest costs.

\FloatBarrier
\begin{figure}[h!]
\begin{center}
\includegraphics[width=0.55\textwidth]{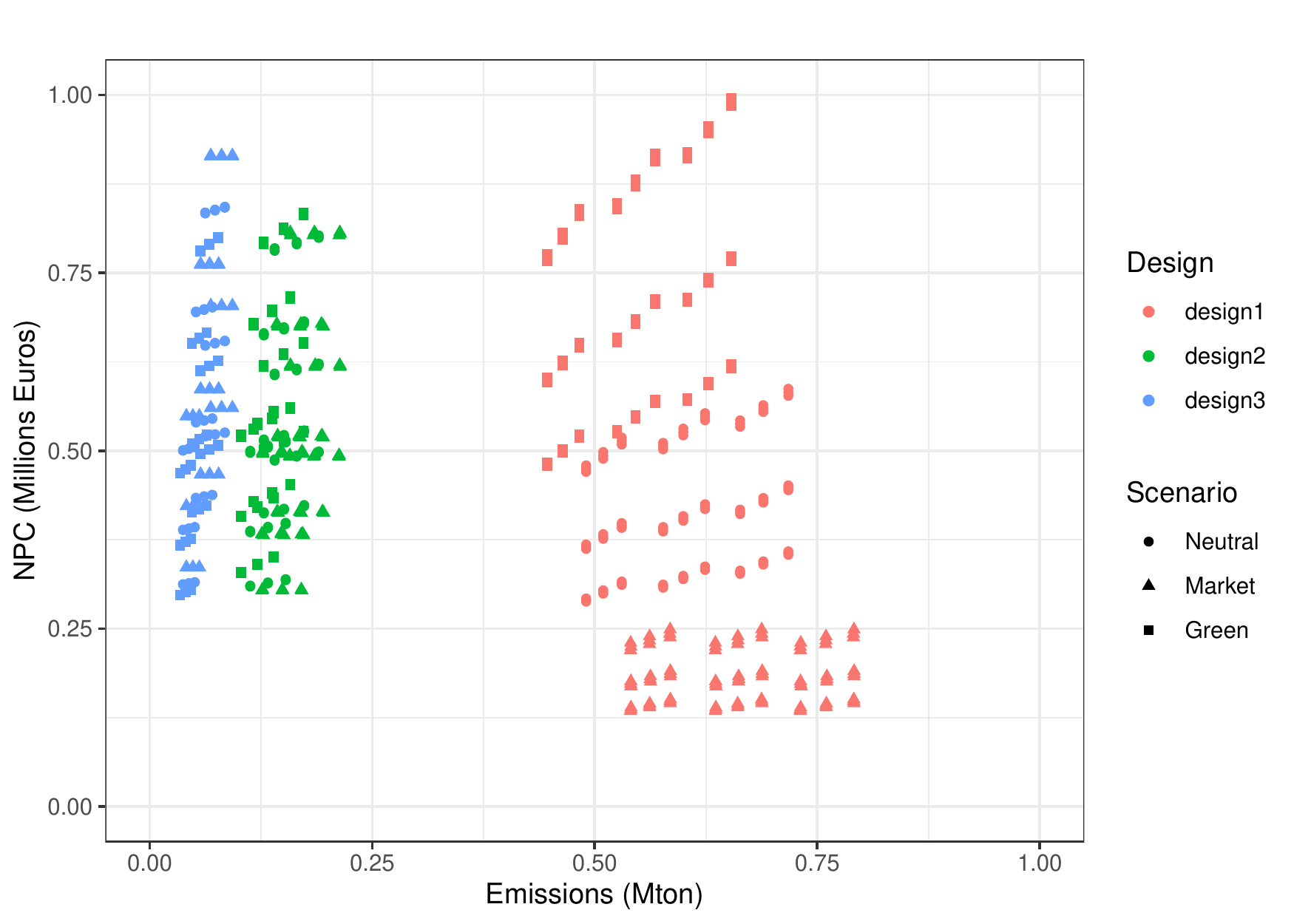}
\end{center}
\caption{Net Present Costs against carbon emissions for each design option and scenario. }
\label{fig:ScatterPlot}
\end{figure}
\FloatBarrier

Employing Algorithms 1 and 2 on the model outputs, we obtain $\tilde{f}_{\hat{f}}(z)$ and $\tilde{F}_{\hat{f}}(z)$. To apply equal importance to both outputs, we scale the data on $[0, 1]$ and generate a uniform set $S$ across $[0, 1]\times [0, 1]$ of size $N=2500$. To produce a smooth representation of the DR and its cdf, we set $M=5000$.

\FloatBarrier
\begin{figure}[h!]
\begin{center}
\includegraphics[width=1\textwidth]{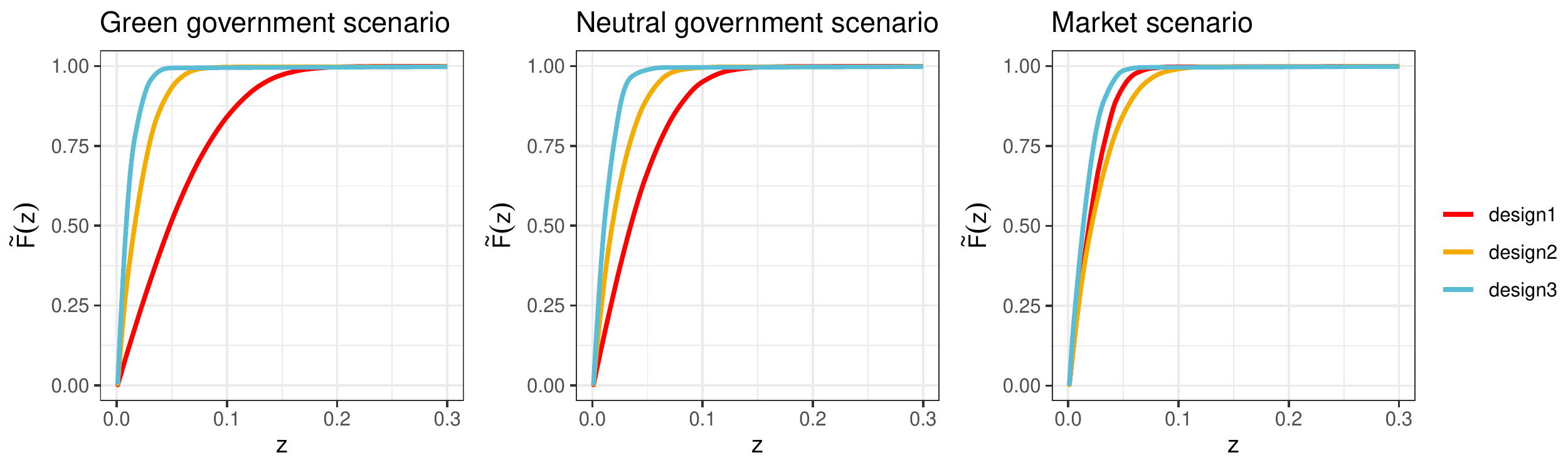}
\vspace{-2em} %%
\end{center}
\caption{Empirical DR cdfs  $\tilde{F}_{\hat{f}}(z)$ for all three design options plotted together for each individual scenario.}
\label{fig:CDF_HeatExample}
\end{figure}
\FloatBarrier

Plots of the empirical DR cdfs $\tilde{F}_{\hat{f}}(z)$ are shown in Figure \ref{fig:CDF_HeatExample}. A feature here is that, under the green and neutral scenarios, the cdf for design option 3 lies above that for design option 2, which lies above that for design option~1 whilst, under the market scenario, the difference of DR cdfs for design option 1 and design option 2 contains a zero, which indicates that the two distributions cannot be compared according to $\preceq$. We conclude that, under all three scenarios, the (unknown) distribution function associated with design option 3 majorises the cdfs for both design options 1 and 2. Therefore, for the outputs considered, design option 3 is less uncertain (more robust) than the alternatives.

Table \ref{tab:entropies} provides values of Shannon and Tsallis entropies computed using the DR pdfs for each design option under the three scenarios. We observe that the total orderings imposed by the entropies on the distribution functions are in agreement with the majorisation orderings in Figure \ref{fig:CDF_HeatExample} under the green and neutral scenarios. This result is supported by condition (B2) in Section \ref{sec:cont_major}, whereas, under the market scenario, both entropy measures provide us with the total orderings, which are different to each other. In particular, the lowest Shannon entropy is obtained for design option 3 followed by design option 1 and design option 2, whereas, for Tsallis entropy, the value for design option 2 is lower that for design option 1.
 
\begin{table}[h!]
\caption{Entropies computed using DR pdfs for each design option under three scenarios.}  
\begin{center}
\begin{small}
\begin{tabular}{c|lll|lll}
& \multicolumn{3} {l|}{\textbf{Shannon entropy}}& \multicolumn{3} {l}{\textbf{Tsallis entropy with $\gamma = 1$ }}\\
\midrule
\textbf{Option} & \textbf{Green} & \textbf{Neutral} & \textbf{Market}& \textbf{Green} & \textbf{Neutral} & \textbf{Market}\\
design 1 & 7.45 & 7.28 & 6.60 & 0.927 & 0.944 & 0.931\\
design 2 & 6.49 & 6.62 & 6.80 & 0.920 & 0.922 & 0.928\\
design 3 & 5.84 & 5.97 & 6.11 & 0.904 & 0.902 & 0.896\\
\midrule
\end{tabular}
\end{small}
\end{center}
\label{tab:entropies}
\vspace{-1em}
\end{table}

We now demonstrate the uncertainty tools from Section \ref{sec:algebra} in order to combine the uncertainty under different scenarios and produce orderings of the design options. In particular, under each design option, we find the maximum of the empirical cdfs associated with individual scenarios to obtain an approximation to $\tilde{F}_1(z)\lor\tilde{F}_2(z)\lor\tilde{F}_3(z)$.  This is shown in the left panel of Figure \ref{fig:maxminDR} and can be considered to represent an  optimistic (more certain) approach. We find that design option 3 majorises the other design options. 
\FloatBarrier
\begin{figure}[h!]
\begin{center}
\includegraphics[width=.7\textwidth]{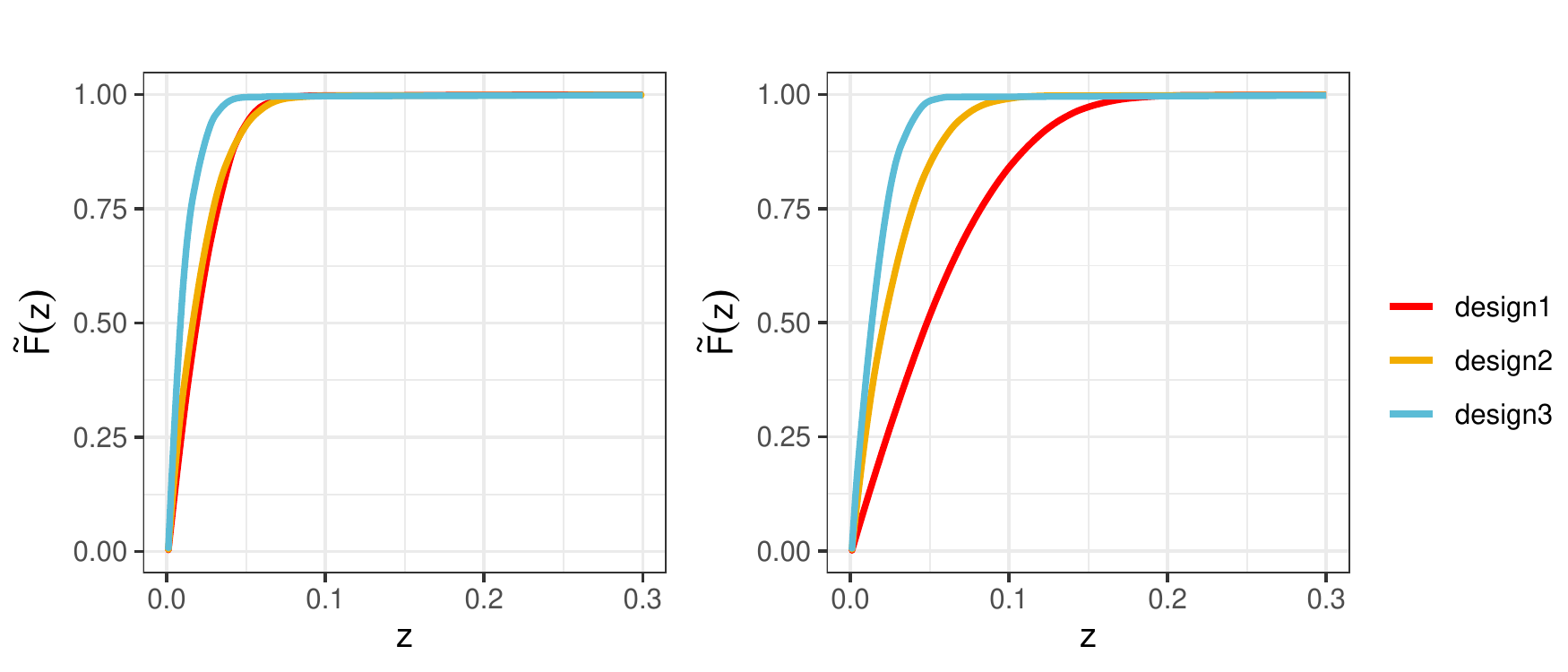}
\vspace{-1em} %%
\end{center}
\caption{\textit{Left panel}: representation of $\max(\tilde{F}_1(z), \tilde{F}_2(z), \tilde{F}_3(z))$. \textit{Right panel}: representation of $\min(\tilde{F}_1(z), \tilde{F}_2(z), \tilde{F}_3(z))$.}
\label{fig:maxminDR}
\end{figure}
\FloatBarrier

We also produce an approximation to $\tilde{F}_1(z)\land\tilde{F}_2(z)\land\tilde{F}_3(z)$, which corresponds to  the pessimistic (less certain) approach.  The results are shown in the right panel of Figure \ref{fig:maxminDR} in which we obtain the minimum of the empirical cdfs associated with individual scenarios. In this case, we observe a clear ordering between design options: design option 3 majorises design option 2, which majorises design option 1. Under both the pessimistic and optimistic outlooks, we conclude that design option 3 is less uncertain than the two alternatives.

\begin{figure}[h!]
\begin{center}
\includegraphics[width=1\textwidth]{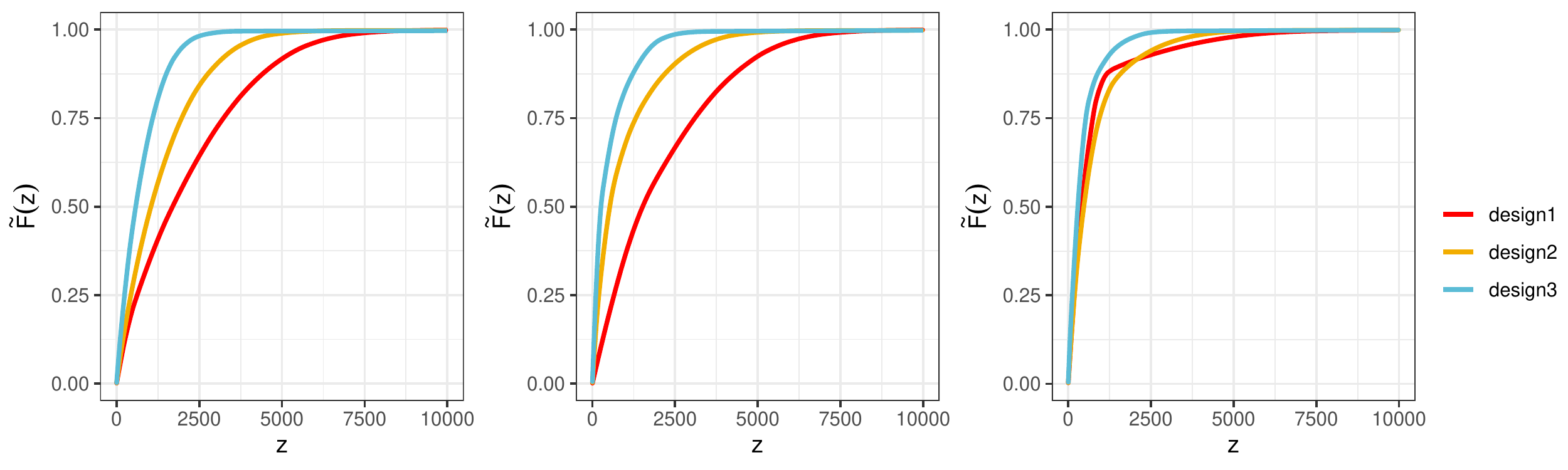}
\vspace{-2em} %%
\end{center}
\caption{Cdfs from inverse mixing with different weightings on each scenario: \textit{Left panel}: equal weights on each scenario. \textit{Central panel}: $\alpha_{G}=0.7$, $\alpha_{N}=0.2$ and $\alpha_{M}=0.1$. \textit{Right panel}: $\alpha_{G}=0.05$, $\alpha_{N}=0.05$ and $\alpha_{M}=0.9$}
\label{fig:HeatInverseMixing}
\end{figure}
 
In practice, the proposed uncertainty tools provide experts and analysts with additional ways to express their expert judgements. For instance, the weights in inverse mixing represent the probabilities of the occurrence of each scenario. Let $\alpha_{G}, \alpha_{N}$ and $\alpha_{M}$ be the weights applied to the Green, Neutral and Market scenarios, respectively.  We consider three cases: (i) equal weights, (ii) $\alpha_G=0.7$, $\alpha_M=0.15$ and $\alpha_N=0.15$ and (iii) $\alpha_G=0.05$, $\alpha_M=0.9$ and $\alpha_N=0.05$. The cdf from inverse mixing for each of these cases is shown in Figure \ref{fig:HeatInverseMixing}.  In cases (i) and (ii), there are clear orderings in which the cdf of design option 3 lies above the cdf of design option 2 which lies above that of design option 1.  In case (iii), however, there is no ordering between the empirical cdfs since the cdfs for design options 1 and 2 cross. However, the cdf associated with design option 3 majorises the cdfs for both design options 1 and 2 and we conclude that design option 3 is the least risky option in all three cases. It is important to note that, whilst the above results provide useful guidance for comparing uncertainty, the uncertainty is only one aspect of such decisions and one would want to take into account the actual costs and carbon emissions (rather than just their variability) in each case. However, here we have demonstrated majorisation to be an intuitive approach to comparing uncertainty and ultimately aiding informed decisions in such settings.

\section{Concluding remarks}\label{sec:conclusion}
The concept of uncertainty is the subject of much discussion, particularly at the technical interface between scientific modelling and statistics. We suggest that majorisation, which only compares the rank order of probability mass, continuous or discrete, provides a valuable form of uncertainty. We have shown that any two distributions can be compared, and consider this to be a principal contribution of the paper. 
We demonstrated this approach to assess the uncertainty with examples from well known distributions and in applications of climate projections and energy systems. The algorithms are straightforward and were introduced to enhance the understanding of the concept of majorisation. The idea presented is that a candidate for a wider framework is a stochastic ordering for which most, if not all, types of entropy are order preserving. 

When events cannot be compared with respect to the majorisation ordering,  questions of relative uncertainty are unanswerable: if events can be compared, the comparison is stronger as it can be made for fewer pairs of events. We believe that this strong condition deserves to be a candidate for a universal version of what is meant by `more certain' or `more uncertain.'

Extensions to our approach to uncertainty include: developing computationally efficient and scalable algorithms to perform empirical decreasing rearrangements; using majorisation in sensitivity analysis, that is, the study of the propagation of variability through systems from input to output; further exploring the properties of the uncertainty ring and lattice and the connection between the two algebraic structures.

\section*{Acknowledgements}
We would like to thank Chris Dent (Edinburgh), Jim Smith (Warwick) and Peter Challenor (Exeter) for their senior support. Authors three and four acknowledge the the EU grant ReUseHeat, ID: 767429, conducted under H2020-EU 3.3.1.

\bibliographystyle{plain}
\bibliography{References}
\end{document}